\documentclass[12pt]{article}
\usepackage{graphics}
\usepackage{latexsym}
\usepackage{amsmath}
\usepackage{amsfonts}
\def\ds{\displaystyle}

\def\bar{\overline}

\def\bar{\overline}
\def\mbf{\mathbf}

\def\bar{\overline}

\def\Im{{\rm Im}}

\def\Im{{\rm Im}}
\newcommand{\R}{{\mathbb R}}
\newcommand{\C}{{\mathbb C}}

\newtheorem{lemma}{Lemma}[section]
\newtheorem{theorem}{Theorem}[section]

\newenvironment{proof}{\noindent {\it Proof}:}{$\Box$}

\usepackage{amssymb}
\usepackage{amsmath}
\usepackage{graphicx}

\begin{document}

\title{Extrapolation-based implicit-explicit general linear methods}

\author{A. Cardone,%
\thanks{Dipartimento di Matematica,
Universit\`{a} degli studi di Salerno, Fisciano (Sa),
84084 Italy,
\mbox{e-mail}: ancardone@unisa.it.
The work of this author was supported by travel fellowship from
the Department of Mathematics, University of Salerno.
} \
Z. Jackiewicz,%
\thanks{Department of Mathematics, Arizona State University,
Tempe, Arizona 85287,
\mbox{e-mail}: jackiewicz@asu.edu, 
and AGH University of Science and Technology, Krak\'ow, Poland.
} \ A. Sandu,%
\thanks{Department of Computer Science, Virginia Polytechnic Institute \& State University,
Blacksburg, Virginia 24061,
\mbox{e-mail}: sandu@cs.vt.edu.
} \ and H. Zhang%
\thanks{Department of Computer Science, Virginia Polytechnic Institute \& State University,
Blacksburg, Virginia 24061,
\mbox{e-mail}: zhang@vt.edu.}
}

\maketitle

\textbf{Abstract}
For many systems of differential equations modeling problems in science and
engineering, there are natural splittings of the right hand side into two
parts, one non-stiff or mildly stiff, and the other one stiff.
For such systems implicit-explicit (IMEX) integration combines an explicit scheme for the non-stiff part with an implicit scheme for the stiff part.
In a recent series of papers two of the authors (Sandu and Zhang) have 
developed  IMEX GLMs, a family of implicit-explicit schemes based on general linear methods. 
It has been shown that, due to their high stage order, IMEX GLMs require no additional coupling order conditions, and are not marred by order reduction.
This work develops a new extrapolation-based approach to construct practical IMEX GLM pairs of high order.
We look for methods with large absolute stability region, assuming that the
implicit part of the method is $A$- or $L$-stable. We provide examples of
IMEX GLMs with optimal stability properties. Their application to a two dimensional test problem confirms the theoretical findings.
\vspace{3mm}

\vspace{2mm}

\newpage

\setcounter{equation}{0}
\setcounter{figure}{0}
\setcounter{table}{0}


\section{Introduction} \label{sec1}
Many practical problems in science and engineering are modeled by large systems
of ordinary differential equations (ODEs) which arise from discretization in space
of partial differential equations (PDEs) by finite difference methods,
finite elements or finite volume methods, or pseudospectral methods.
For such systems there are often natural splittings of the right hand sides
of the differential systems into two parts, one of which is non-stiff or
mildly stiff, and suitable for explicit time integration, and the other part is stiff,
and suitable for implicit time integration. Such systems can be written in the form
\begin{equation} \label{eq1.1}
\left\{
\begin{array}{ll}
y'(t)=f\big(y(t)\big)+g\big(y(t)\big), & t\in[t_0,T], \\
y(t_0)=y_0,
\end{array}
\right.
\end{equation}
where $f(y)$ represents the non-stiff processes, for example advection,
and $g(y)$ represents stiff processes, for example diffusion or chemical reaction,
in semi-discretization of advection-diffusion-reaction equations \cite{hv03}.

Implicit-explicit (IMEX) integration approach discretizes the non-stiff part $f(y)$ is
with an explicit method, and the stiff part $g(y)$ with an implicit, stable method.
This strategy seeks to ensure the numerical stability of the solution of \eqref{eq1.1}
while reducing the amount of implicitness, and therefore the  overall computational effort.
IMEX multistep methods were introduced by Crouzeix \cite{cro80}
and Varah \cite{var80} and further analyzed in \cite{arw95,fhv97}.
IMEX Runge-Kutta methods have been investigated in \cite{ars97,cfn01,kc03,pr00,pr05,zho96}.

In a recent series of papers the last two authors and their collaborators have
proposed the new IMEX GLM family of
implicit-explicit schemes based on general linear methods.
A general formalism for partitioned GLMs and their order conditions
was developed by Zhang and Sandu \cite{zs13}. The partitioned method formalism was then used to construct IMEX GLMs. The
starting and ending procedures, linear stability, and stiff convergence properties of the new family have been analyzed. Zhang and Sandu
examined practical methods of second
order in \cite{zs12} and of third order in \cite{zs13}.
A class of IMEX two step Runge-Kutta (TSRK)
methods was proposed by Zharovski and Sandu \cite{zhasan13}.

The results in \cite{zs12,zs13,zhasan13} prove that the general linear framework is well suited for the construction of multi-methods. Specifically, owing to the high stage orders, no coupling conditions are needed to ensure the order of accuracy of the partitioned GLM \cite{zs13}. In addition, it has been shown that IMEX GLMs are particularly attractive for solving stiff problems, where other multistage methods may suffer from order reduction \cite{zs13}.

This paper extends our previous work \cite{zs12,zs13,zhasan13} and develops a new extrapolation-based approach
for the construction of practical IMEX GLM schemes of high order and high stage order.

The organization of this paper is as follows.
General linear methods and the implicit-explicit variants are reviewed in Section \ref{sec:imex-glm}.
The new extrapolation-based IMEX GLMs are derived in Section \ref{sec:eximex}, and their order conditions are presented.
The stability analysis is performed in Section \ref{sec:stability} and specific methods are constructed in Section \ref{sec:construction}.
Numerical experiments are presented in Section \ref{sec:numerics}, and Section \ref{sec:conclusions} gives some concluding remarks and plans for future work.

\section{Implicit-explicit general linear methods}\label{sec:imex-glm}
\setcounter{equation}{0}
\setcounter{figure}{0}
\setcounter{table}{0}

In this section we briefly review GLMs and the IMEX GLM family.

The GLMs for ODEs were introduced by Burrage and Butcher \cite{bb80}
and further investigated in \cite{bur95,but87,but03,cj12,cjm,hlw02,hnw93,hw96}.
We also refer the reader to the review article \cite{but06a} and the recent
monograph \cite{jac09} and references therein.

A diagonally implicit GLM for (\ref{eq1.1}) is defined by
\begin{equation} \label{eq1.2}
\left\{
\begin{array}{lcl}
Y_i^{[n+1]} & = & h\ds\sum_{j=1}^ia_{ij}\Big(f\big(Y_j^{[n+1]}\big)
+g\big(Y_j^{[n+1]}\big)\Big)+\ds\sum_{j=1}^ru_{ij}y_j^{[n]},
\quad i=1,2,\ldots,s \\ [3mm]
y_i^{[n+1]} & = & h\ds\sum_{j=1}^sb_{ij}\Big(f\big(Y_j^{[n+1]}\big)
+g\big(Y_j^{[n+1]}\big)\Big)+\ds\sum_{j=1}^rv_{ij}y_j^{[n]},
\quad i=1,2,\ldots,r, \\ [3mm]
\end{array}
\right.
\end{equation}
$n=0,1,\ldots,N-1$.
Here, $N$ is a positive integer, $h=(T-t_0)/N$, $t_n=t_0+nh$, $n=0,1,\ldots,N$,
$Y_i^{[n+1]}$ are approximations of stage order $q$ to $y(t_n+c_ih)$,
i.e.,
\begin{equation} \label{eq1.3}
Y_i^{[n+1]}=y(t_n+c_ih)+O(h^{q+1}),
\quad
i=1,2,\ldots,s,
\end{equation}
$y_i^{[n]}$ are approximations of order $p$ to the
linear combinations of the derivatives of the solution $y$ at the point $t_n$,
i.e.,
\begin{equation} \label{eq1.4}
y_i^{[n]}=\ds\sum_{k=0}^pq_{ik}h^ky^{(k)}(t_n)+O(h^{p+1}),
\quad
i=1,2,\ldots,r,
\end{equation}
and $y$ is the solution to (\ref{eq1.1}).
These methods can be characterized by the abscissa vector
$\mbf{c}=[c_1,\ldots,c_s]^T$, the coefficient matrices
$\mbf{A}=[a_{ij}]\in\R^{s\times s}$,
$\mbf{U}=[a_{ij}]\in\R^{s\times r}$,
$\mbf{B}=[a_{ij}]\in\R^{r\times s}$,
$\mbf{V}=[a_{ij}]\in\R^{r\times r}$,
the vectors $\mbf{q}_0,\ldots,\mbf{q}_s\in\R^{r}$ defined by $\mbf{q}_i = [q_{j,i}]_{1 \le j \le r}$,
and four integers: the order $p$, the stage order $q$,
the number of external approximations $r$, and the number
of stages or internal approximations $s$.

The method (\ref{eq1.2}) can be written in a compact form
\begin{equation} \label{underlying-GLM}
\left\{
\begin{array}{l}
Y^{[n+1]}=h(\mbf{A}\otimes \mbf{I})\Big(f\big(Y^{[n+1]}\big)+g\big(Y^{[n+1]}\big)\Big)
+(\mbf{U}\otimes \mbf{I})y^{[n]}, \\ [3mm]
y^{[n+1]}=h(\mbf{B}\otimes \mbf{I})\Big(f\big(Y^{[n+1]}\big)+g\big(Y^{[n+1]}\big)\Big)
+(\mbf{V}\otimes \mbf{I})y^{[n]},
\end{array}
\right.
\end{equation}
$n=0,1,\ldots,N-1$, and the relation (\ref{eq1.4}) takes the form
\begin{equation} \label{eq1.6}
y^{[n]}=\ds\sum_{k=0}^p\mbf{q}_kh^ky^{(k)}(t_n)+O(h^{p+1}).
\end{equation}

Applying \eqref{eq1.2} to the basic test equation
$y'(t)=\lambda y(t)$, $t\geq 0$, $\lambda \in \C$, leads to the recurrence
equation
$$
y^{[n+1]}=\mbf{S}(z)y^{[n]},
\quad
n=0,1,\ldots,
$$
$z=h\lambda$, with the stability matrix given by
\begin{equation} \label{eq1.7}
\mbf{S}(z)=\mbf{V}+z\mbf{B}(\mbf{I}-z\mbf{A})^{-1}\mbf{U}.
\end{equation}
We also define the stability polynomial $\eta(w,z)$ by
\begin{equation} \label{eq1.8}
\eta(w,z)=\det\big(w\mbf{I}-\mbf{S}(z)\big).
\end{equation}
The region of absolute stability of the method (\ref{eq1.2}) is
the subset of the complex plane
\begin{equation} \label{eq1.9}
\mathcal{A}=\big\{z\in\C: \ \textrm{all roots} \ w_i(z) \ \textrm{of} \ \eta(w,z) \
\textrm{are in the unit circle}\big\}.
\end{equation}
The traditional concepts of $A(\alpha)$-stability,  $A$-stability, and $L$-stability
apply directly to GLMs via \eqref{eq1.9}.

In this paper we will examine only methods of high stage order, i.e., methods
where $q=p-1$ or $q=p$. It has been shown in \cite{bj93,but93,jac09} that
%
%
the GLM (\ref{eq1.2}) has order $p$ and stage order $q=p$ or $q=p-1$ if and only if
\begin{eqnarray}
\label{stage-order-cond}
\mbf{c}^k-k\,\mbf{A}\, \mbf{c}^{k-1} - k!\, \mbf{U}\, \mbf{q}_k &=& 0\,, \quad k=0,1,\dots,q\,,  \quad \textnormal{and}\\
\label{order-cond}
\sum_{\ell=0}^k \frac{k!}{\ell!}\, \mbf{q}_{k-\ell} - k\,\mbf{B}\, \mbf{c}^{k-1} - k!\, \mbf{V}\, \mbf{q}_k &=& 0\,, \quad k=0,1,\dots,p\,.
\end{eqnarray}
%

%

An IMEX-GLM \cite[Definition 4]{zhasan13} has the form
\begin{equation} \label{imex-glm}
\left\{
\begin{array}{l}
Y^{[n+1]}=h (\mbf{A}^{\rm exp}\otimes \mbf{I})\,f\big(Y^{[n+1]}\big)+h (\mbf{A}^{\rm imp}\otimes \mbf{I})\,g\big(Y^{[n+1]}\big)
+(\mbf{U}\otimes \mbf{I})y^{[n]}, \\ [3mm]
y^{[n+1]}=h(\mbf{B}^{\rm exp}\otimes \mbf{I})\,f\big(Y^{[n+1]}\big)+h(\mbf{B}^{\rm imp}\otimes \mbf{I})\, g\big(Y^{[n+1]}\big)
+(\mbf{V}\otimes \mbf{I})y^{[n]},
\end{array}
\right.
\end{equation}
where $\mbf{A}^{\rm exp}$, $\mbf{B}^{\rm exp}$ correspond to the explicit part and  $\mbf{A}^{\rm imp}$, $\mbf{B}^{\rm imp}$ to the implicit part. The methods share the same
abscissa $\mbf{c}^{\rm exp}=\mbf{c}^{\rm imp}$, which makes \eqref{imex-glm} internally consistent \cite[Definition 2]{zhasan13}.
The methods also share the same coefficient matrices $\mbf{U}^{\rm exp}=\mbf{U}^{\rm imp}=\mbf{U}$ and $\mbf{V}^{\rm exp}=\mbf{V}^{\rm imp}=\mbf{V}$. The coefficients $\mbf{q}^{\rm exp}_k$, $\mbf{q}^{\rm imp}_k$ in \eqref{eq1.6} can be different,
which means that the implicit and explicit components use different initialization and termination procedures.
An IMEX-GLM \eqref{imex-glm} is a special case of a partitioned GLM \cite[Definition 1]{zhasan13};
while in \eqref{imex-glm} the right hand side is split in two components, stiff and nonstiff, a partitioned GLM
allows for splitting in an arbitrary number of components.

%
It has been shown in \cite[Theorem 2]{zhasan13} that
an internally consistent partitioned GLM (and, in particular, the IMEX GLM \eqref{imex-glm})
has order  $p$ and stage order $q \in \{p-1,p\}$
if and only if each component method $\left(\mathbf{A}^{\rm exp},\mathbf{B}^{\rm exp},\mathbf{U},\mathbf{V}\right)$
and $\left(\mathbf{A}^{\rm imp},\mathbf{B}^{\rm imp},\mathbf{U},\mathbf{V}\right)$ has order $p$  and stage order $q$.
%
We note that no additional ``coupling'' conditions are needed for the IMEX GLM (i.e.,
no order conditions that contain coefficients of both the implicit and the explicit schemes).

\section{Extrapolation-based IMEX GLMs} \label{sec:eximex}
\setcounter{equation}{0}
\setcounter{figure}{0}
\setcounter{table}{0}
\subsection{Method formulation} \label{sec2}

In this section we derive the new extrapolation-based IMEX GLMs.
Consider the following extrapolation formula depending on
stage values $Y_k^{[n]}$ and $Y_k^{[n+1]}$ at two consecutive steps
\begin{equation} \label{eq1.11}
f_j^{[n+1]}=\ds\sum_{k=1}^s\alpha_{jk}f\big(Y_k^{[n]}\big)
+\ds\sum_{k=1}^{j-1}\beta_{jk}f\big(Y_k^{[n+1]}\big),
\quad
j=1,2,\ldots,s.
\end{equation}
Substituting $f_j^{[n+1]}$ in (\ref{eq1.11}) for $f\big(Y_j^{[n+1]}\big)$ in (\ref{eq1.2}) leads to the proposed class of
extrapolation-based IMEX GLMs. The simple example of IMEX method consisting of the explicit
Euler method combined with the $A$-stable implicit $\theta$-method
corresponding to $\theta\geq 1/2$ is presented in \cite{hv03}.

 Substituting (\ref{eq1.11}) into (\ref{eq1.2}) leads to
$$
\begin{array}{lcl}
Y_i^{[n+1]} & = & h\ds\sum_{j=1}^i\ds\sum_{k=1}^sa_{ij}\alpha_{jk}f\big(Y_k^{[n]}\big)
+h\ds\sum_{j=1}^i\ds\sum_{k=1}^{j-1}a_{ij}\beta_{jk}f\big(Y_k^{[n+1]}\big) \\ [3mm]
& + &
h\ds\sum_{j=1}^ia_{ij}g\big(Y_j^{[n+1]}\big)+\ds\sum_{j=1}^ru_{ij}y_j^{[n]},
\quad i=1,2,\ldots,s, \\ [3mm]
y_i^{[n+1]} & = & h\ds\sum_{j=1}^s\ds\sum_{k=1}^sb_{ij}\alpha_{jk}f\big(Y_k^{[n]}\big)
+h\ds\sum_{j=1}^s\ds\sum_{k=1}^{j-1}b_{ij}\beta_{jk}f\big(Y_k^{[n+1]}\big) \\ [3mm]
& + &
h\ds\sum_{j=1}^ib_{ij}g\big(Y_j^{[n+1]}\big)+\ds\sum_{j=1}^rv_{ij}y_j^{[n]},
\quad i=1,2,\ldots,r, \\ [3mm]
\end{array}
$$
$n=0,1,\ldots,N-1$.
Changing the order of summation in the double sums above and then
interchanging the indices $j$ and $k$ we obtain
$$
\begin{array}{lcl}
Y_i^{[n+1]} & = & h\ds\sum_{j=1}^s\ds\sum_{k=1}^ia_{ik}\alpha_{kj}f\big(Y_j^{[n]}\big)
+h\ds\sum_{j=1}^{i-1}\ds\sum_{k=j+1}^ia_{ik}\beta_{kj}f\big(Y_j^{[n+1]}\big) \\ [3mm]
& + &
h\ds\sum_{j=1}^ia_{ij}g\big(Y_j^{[n+1]}\big)+\ds\sum_{j=1}^ru_{ij}y_j^{[n]},
\quad i=1,2,\ldots,s,\\ [3mm]
y_i^{[n+1]} & = & h\ds\sum_{j=1}^s\ds\sum_{k=1}^sb_{ik}\alpha_{kj}f\big(Y_j^{[n]}\big)
+h\ds\sum_{j=1}^{s-1}\ds\sum_{k=j+1}^sb_{ik}\beta_{kj}f\big(Y_j^{[n+1]}\big) \\ [3mm]
& + &
h\ds\sum_{j=1}^sb_{ij}g\big(Y_j^{[n+1]}\big)+\ds\sum_{j=1}^rv_{ij}y_j^{[n]},
\quad i=1,2,\ldots,r.\\
\end{array}
$$
These relations lead to IMEX GLMs of the form
\begin{equation} \label{eq2.1}
\left\{
\begin{array}{lcl}
Y_i^{[n+1]} & = & h\ds\sum_{j=1}^s\bar{a}_{ij}f\big(Y_j^{[n]}\big)
+h\ds\sum_{j=1}^{i-1}a_{ij}^{*}f\big(Y_j^{[n+1]}\big) \\ [3mm]
& + &
h\ds\sum_{j=1}^ia_{ij}g\big(Y_j^{[n+1]}\big)
+\ds\sum_{j=1}^ru_{ij}y_j^{[n]},
\quad i=1,2,\ldots,s,\\ [3mm]
y_i^{[n+1]} & = & h\ds\sum_{j=1}^s\bar{b}_{ij}f\big(Y_j^{[n]}\big)
+h\ds\sum_{j=1}^{s-1}b_{ij}^{*}f\big(Y_j^{[n+1]}\big) \\ [3mm]
& + &
h\ds\sum_{j=1}^sb_{ij}g\big(Y_j^{[n+1]}\big)
+\ds\sum_{j=1}^rv_{ij}y_j^{[n]},
\quad i=1,2,\ldots,r,
\end{array}
\right.
\end{equation}
$n=0,1,\ldots,N-1$,
where the coefficients $\bar{a}_{ij}$, $a_{ij}^{*}$,
$\bar{b}_{ij}$, and $b_{ij}^{*}$ are defined by
$$
\bar{a}_{ij}=\ds\sum_{k=1}^ia_{ik}\alpha_{kj},
\quad
a_{ij}^{*}=\ds\sum_{k=j+1}^ia_{ik}\beta_{kj},
\quad
\bar{b}_{ij}=\ds\sum_{k=1}^sb_{ik}\alpha_{kj},
\quad
b_{ij}^{*}=\ds\sum_{k=j+1}^sb_{ik}\beta_{kj}.
$$
In matrix notation
$$
\bar{\mbf{A}}=[\bar{a}_{ij}]\in\R^{s\times s},
\
\mbf{A}^{*}=[a_{ij}^{*}]\in\R^{s\times s},
\
\bar{\mbf{B}}=[\bar{b}_{ij}]\in\R^{r\times s},
\
\mbf{B}^{*}=[b_{ij}^{*}]\in\R^{r\times s}.
$$
with
$$
\bar{\mbf{A}}=\mbf{A} \bf{\alpha},
\quad
\mbf{A}^{*}=\mbf{A} \mathbf{\beta},
\quad
\bar{\mbf{B}}=\mbf{B} \mathbf{\alpha},
\quad
\mbf{B}^{*}=\mbf{B} \mathbf{\beta},
$$
where $\mathbf{\alpha}=[\alpha_{ij}]\in\R^{s\times s}$, $\mathbf{\beta}=[\beta_{ij}]\in\R^{s\times s}$.
Observe that the matrix $\mbf{A}^{*}$ is strictly lower triangular
and that the last column of the matrix $\mbf{B}^{*}$ is zero.

In matrix notation the extrapolation-based IMEX-GLM is defined by:
\begin{eqnarray}
\nonumber
Y^{[n+1]}&=&h(\bar{\mbf{A}}\otimes \mbf{I})f\big(Y^{[n]}\big)
+h(\mbf{A}^{*}\otimes \mbf{I})f\big(Y^{[n+1]}\big) \\
\label{extrap-imex}
&& +h(\mbf{A}\otimes \mbf{I}) g\big(Y^{[n+1]}\big)+(\mbf{U}\otimes \mbf{I})y^{[n]}, \\
\nonumber
y^{[n+1]}&=&h(\bar{\mbf{B}}\otimes \mbf{I})f\big(Y^{[n]}\big)
+h(\mbf{B}^{*}\otimes \mbf{I})f\big(Y^{[n+1]}\big) \\
\nonumber
&&+h(\mbf{B}\otimes \mbf{I})f\big(Y^{[n+1]}\big) +(\mbf{V}\otimes \mbf{I})y^{[n]},
\end{eqnarray}
$n=0,1,\ldots,N-1$.

The explicit part of (\ref{extrap-imex}), obtained for $g(y)=0$, can be represented as a single GLM
extended over two steps from
$t_{n-1}$ to $t_n$ and $t_n$ to $t_{n+1}$, as follows
\begin{equation} \label{explicit}
\left[
\begin{array}{c}
Y^{[n]} \\
Y^{[n+1]} \\
\hline
Y^{[n+1]} \\
y^{[n+1]}
\end{array}
\right]=\left[
\begin{array}{cc|cc}
\mbf{0} & \mbf{0} & \; \mbf{I} & \mbf{0} \\
\bar{\mbf{A}} & \; \mbf{A}^{*} & \; \mbf{0} & \mbf{U} \\
\hline
\bar{\mbf{A}} & \mbf{A}^{*} & \; \mbf{0} & \mbf{U} \\
\bar{\mbf{B}} & \mbf{B}^{*} & \; \mbf{0} & \mbf{V} \\
\end{array}
\right]\left[
\begin{array}{c}
f\big(Y^{[n]}\big) \\
f\big(Y^{[n+1]}\big) \\ \hline
Y^{[n]} \\
y^{[n]}
\end{array}
\right].
\end{equation}
The abscissa vector is
$\mbf{c}^{\rm exp}=[(\mbf{c}-\mbf{e})^T,\mbf{c}^T]^T$.

Similarly, the implicit part of the IMEX scheme (\ref{eq2.1}) corresponding
to $f(y)=0$ assumes the form
\begin{equation} \label{implicit}
\left[
\begin{array}{c}
Y^{[n]} \\
Y^{[n+1]} \\
\hline
Y^{[n+1]} \\
y^{[n+1]}
\end{array}
\right]=\left[
\begin{array}{cc|cc}
\mbf{0} & \mbf{0} & \ \mbf{I} & \mbf{0} \\
\mbf{0} & \mbf{A} & \ \mbf{0} & \mbf{U} \\
\hline
\mbf{0} & \mbf{A} & \ \mbf{0} & \mbf{U} \\
\mbf{0} & \mbf{B} & \ \mbf{0} & \mbf{V} \\
\end{array}
\right]\left[
\begin{array}{c}
g\big(Y^{[n]}\big) \\
g\big(Y^{[n+1]}\big) \\
\hline
Y^{[n]} \\
y^{[n]}
\end{array}
\right].
\end{equation}
This method has the order and stage order of the underlying GLM (\ref{eq1.2}),
since it is the same method.
The abscissa vector is
$\mbf{c}^{\rm imp}=[(\mbf{c}-\mbf{e})^T,\mbf{c}^T]^T$, and therefore the method \eqref{extrap-imex}
is internally consistent.

\subsection{Construction of the interpolant}

We define the local discretization errors
$\eta(t_n+c_jh)$ of the extrapolation formula (\ref{eq1.11})
by the relation
\begin{equation} \label{eq2.11}
\begin{array}{l}
f\big(y(t_n+c_jh)\big)=\ds\sum_{k=1}^s\alpha_{jk}f\big(y(t_{n-1}+c_kh)\big)+ \\ [3mm]
\quad\quad + \
\ds\sum_{k=1}^{j-1}\beta_{jk}f\big(y(t_n+c_kh)\big)+\eta(t_n+c_jh),
\end{array}
\end{equation}
$j=1,2,\ldots,s$.
Letting $\varphi(t)=f(y(t))$ the relation (\ref{eq2.11}) can be written
in the form
$$
\eta(t_n+c_jh)=
\varphi(t_n+c_jh)-\ds\sum_{k=1}^s\alpha_{jk}\varphi\big(t_n+(c_k-1)h\big)
-\ds\sum_{k=1}^{j-1}\beta_{jk}\varphi(t_n+c_kh),
$$
$j=1,2,\ldots,s$.
Expanding $\varphi(t_n+c_jh)$, $\varphi(t_n+(c_k-1)h)$, and $\varphi(t_n+c_kh)$
into Taylor series around $t_n$ we obtain
$$
\eta(t_n+c_jh)=\ds\sum_{l=0}^p
\bigg(\ds\frac{c_j^l}{l!}-\ds\sum_{k=1}^s\alpha_{jk}\ds\frac{(c_k-1)^l}{l!}
-\ds\sum_{k=1}^{j-1}\beta_{jk}\ds\frac{c_k^l}{l!}\bigg)h^l\varphi^{(l)}(t_n)
+O(h^{p+1}).
$$
Assuming that the extrapolation procedure given by (\ref{eq1.11}) has order $p$, i.e.,
\mbox{$\eta(t_n+c_jh)=O(h^{p})$}, leads to the following system of equations for
the interpolation coefficients:
\begin{equation} \label{eq2.12}
\ds\sum_{k=1}^s\alpha_{jk}(c_k-1)^\ell
=c_j^\ell-\ds\sum_{k=1}^{j-1}\beta_{jk}c_k^\ell,
\quad
\ell=0,1,\ldots,p-1,
\quad
j=1,2,\ldots,s.
\end{equation}

In matrix notation we have
\begin{equation}
\label{interpolation-order}
\alpha\, (\mathbf{c}-\mathbf{e})^\ell + \beta \, \mathbf{c}^\ell =\mathbf{c}^\ell\,, \quad \ell=0,1,\ldots,p-1\,.
\end{equation}

\subsection{Stage and order conditions}

\begin{lemma}
Assume that the underlying GLM \eqref{underlying-GLM} has order $p$ and stage order $q=p$ or $q=p-1$, and that the interpolation formula (\ref{eq1.11})
has order $p$ \eqref{interpolation-order}.
Then the explicit method \eqref{explicit} has order $p$ and stage order $q$.
\end{lemma}

\begin{proof}

The method \eqref{explicit} has the coefficients
\[
\mbf{A}^{\rm exp}=\begin{bmatrix}
\mbf{0} & \mbf{0} \\
\bar{\mbf{A}} & \; \mbf{A}^{*}
\end{bmatrix}\,, \quad
\mbf{B}^{\rm exp}=\begin{bmatrix}
\bar{\mbf{A}} & \mbf{A}^{*}  \\
\bar{\mbf{B}} & \mbf{B}^{*}
\end{bmatrix}\,, \quad
\mbf{c}^{\rm exp}=\begin{bmatrix} \mbf{c}-\mbf{e} \\ \mbf{c} \end{bmatrix}\,.
\]
where $\mbf{e}=[1,\ldots,1]^T\in\R^s$,
$$
\mbf{U}^{\rm exp}=\left[
\begin{array}{cc}
\mbf{I} & \mbf{0} \\
\mbf{0} & \mbf{U}
\end{array}
\right],
\quad
\mbf{V}^{\rm exp}=\left[
\begin{array}{cc}
\mbf{0} & \mbf{U} \\
\mbf{0} & \mbf{V}
\end{array}
\right]\,,
$$
and the vectors
$$
{\mbf{q}}_0^{\rm exp}=\left[
\begin{array}{c}
\mbf{e} \\
\mbf{q}_0
\end{array}
\right]\,;
\quad
{\mbf{q}}_i^{\rm exp}=
\begin{bmatrix}
\frac{(\mbf{c}-\mbf{e})^i}{i!} \\
\mbf{q}_i
\end{bmatrix} \,, ~~ i=1,\dots,p\,.
$$
We verify directly that the extrapolation-based explicit method \eqref{explicit} has stage order $q$,
i.e., it satisfies equations \eqref{stage-order-cond}
\begin{eqnarray*}
&& \ds (\mbf{c}^{\rm exp})^k - k\,\mbf{A}^{\rm exp}\, (\mbf{c}^{\rm exp})^{k-1}-k!\,\mbf{U}^{\rm exp}\, \mbf{q}^{\rm exp}_k \\
&& =
\begin{bmatrix}
\mathbf{0} \\
\ds \mbf{c}^k- k\, \mbf{A} \left( \alpha\, (\mbf{c}-\mathbf{e})^{k-1} + \beta\, \mathbf{c}^{k-1} \right)-k!\, \mbf{U}\mbf{q}_k
\end{bmatrix} \\
&& =
\begin{bmatrix}
\mathbf{0} \\
\ds \mbf{c}^k-k\, \mbf{A}\mbf{c}^{k-1}-k!\,\mbf{U}\mbf{q}_k
\end{bmatrix}\quad \{ \textrm{from }\eqref{interpolation-order}\} \\
&& = \begin{bmatrix}
\mathbf{0} \\ \mathbf{0}
\end{bmatrix} \,, \quad k = 0,1,\dots,q \quad \{ \textrm{from }\eqref{stage-order-cond}\} \,.
\end{eqnarray*}
From \eqref{order-cond} we verify that the method \eqref{explicit} has order $p$
\begin{eqnarray*}
&& \sum_{\ell=0}^k \frac{k!}{\ell!}\, \mbf{q}^{\rm exp}_{k-\ell} - k\,\mbf{B}^{\rm exp}\, (\mbf{c}^{\rm exp})^{k-1} - k!\, \mbf{V}^{\rm exp}\, \mbf{q}^{\rm exp}_k \\
&& =  \sum_{\ell=0}^k \frac{k!}{\ell!}\,\begin{bmatrix} \ds \frac{(\mbf{c}-\mbf{e})^{k-\ell}}{(k-\ell)!} \\ \ds \mbf{q}_{k-\ell} \end{bmatrix}
- k\, \begin{bmatrix} \mathbf{A}\, \left( \alpha\, (\mathbf{c}-\mathbf{e})^{k-1} + \beta\, \mathbf{c}^{k-1}   \right) \\   \mathbf{B}\, \left( \alpha\, (\mathbf{c}-\mathbf{e})^{k-1} + \beta\, \mathbf{c}^{k-1}   \right) \end{bmatrix}
- k!\,\begin{bmatrix}  \mathbf{U}\, \mathbf{q}_k \\  \mathbf{V}\, \mathbf{q}_k \end{bmatrix} \\
&& = \begin{bmatrix}
\ds\sum_{\ell=0}^k \frac{k!}{\ell!}  \frac{(\mbf{c}-\mbf{e})^{k-\ell}}{(k-\ell)!} - k\,\mbf{A}\mbf{c}^{k-1}- k!\, \mathbf{U}\, \mathbf{q}_k \\
\ds\sum_{\ell=0}^k \frac{k!}{\ell!} \mbf{q}_{k-\ell} - k\,\mbf{B}\mbf{c}^{k-1}- k!\, \mathbf{V}\, \mathbf{q}_k
\end{bmatrix} \quad \{ \textrm{from }\eqref{interpolation-order} \} \\
&& = \begin{bmatrix}
\mathbf{0} \\ \mathbf{0}
\end{bmatrix} \,, \quad k = 0,1,\dots,p  \quad \{ \textrm{from }\eqref{stage-order-cond}, \eqref{order-cond},  \textrm{ and } \eqref{combinatorial} \}\,.
\end{eqnarray*}
For the first component we have used the fact that
\begin{equation}
\label{combinatorial}
\mbf{c}^\ell = \bigl( (\mbf{c}-\mbf{e}) + \mbf{e} \bigr)^\ell = \sum_{\ell=0}^k \frac{\ell!}{k!\,(\ell-k)! } (\mbf{c}-\mbf{e})^{k-\ell}\,.
\end{equation}
\flushright{$\square$}
\end{proof}

To analyze the order and stage order of IMEX GLMs (\ref{eq2.1}) we will
impose some conditions on the local discretization errors of the internal
and external stages of the underlying GLM (\ref{eq1.2}) and on the accuracy of
the extrapolation procedure (\ref{eq1.11}).

We have the following result.
\begin{theorem} \label{th2.3}
Assume that the underlying GLM (\ref{eq1.2}) has order $p$ and stage order $q=p$ or $q=p-1$,
and that the extrapolation procedure (\ref{eq1.11}) has order $p$. Then the IMEX GLM (\ref{eq2.1})
has order $p$ and stage order $q=p$ or $q=p-1$.
\end{theorem}
\begin{proof}

The method (\ref{extrap-imex}) is an IMEX GLM  of the form \eqref{imex-glm} \cite[Definition 1]{zs13}.
The explicit (\ref{explicit}) and implicit (\ref{implicit}) components have the same order $p$ and stage order $q$, and they share
the same abscissa vector and the same coefficients
\[
\mbf{c}^{\rm exp} = \mbf{c}^{\rm imp}\,,
\quad
\mbf{U}^{\rm imp}=\mbf{U}^{\rm exp}\,,
\quad
\mbf{V}^{\rm imp}=\mbf{V}^{\rm exp}\,.
\]
The result follows directly from \cite[Theorem 2]{zs13}.
\flushright{$\square$}
\end{proof}

\subsection{Prothero-Robinson convergence of IMEX GLMs} \label{sec3}
The extrapolation IMEX-GLM schemes (\ref{eq2.1})
do not suffer from order reduction phenomenon when applied to stiff
systems of differential equations.
Following \cite{but87,zs13,zhasan13} we consider the Prothero-Robinson
(PR) \cite{pr74} test problem of the form
\begin{equation} \label{eq3.1}
\left\{
\begin{array}{ll}
y'(t)=\mu\big(y(t)-\phi(t)\big)+\phi'(t), & t\geq 0, \\
y(0)=\phi(0),
\end{array}
\right.
\end{equation}
where $\mu\in\C$ has a large and negative real part and $\phi(t)$ is a
slowly varying function.
The solution to (\ref{eq3.1}) is $y(t)=\phi(t)$.
The IMEX scheme (\ref{eq2.1}) is said to be PR-convergent if the application
of (\ref{eq2.1}) to the equation (\ref{eq3.1}) leads to the numerical
solution $y^{[n]}$ whose global error satisfies
$$
\left\|y^{[n]}-\ds\sum_{k=0}^p\mbf{q}_kh^ky^{(k)}(t_n)\right\|=\mathcal{O}(h^p)
\quad
\textrm{as}
\quad
h\rightarrow 0
\quad
\textrm{and}
\quad
h\mu\rightarrow -\infty.
$$
We have the following result.

%
%
\begin{theorem} \label{th3.1}
Assume that the implicit GLM (\ref{eq1.2}) has order
$p$ and stage order $q=p-1$ or $q=p$, and that the extrapolation
formula (\ref{eq1.11}) has order $p$.
Then the IMEX scheme (\ref{eq2.1}) is PR-convergent with
order $\min(p,q)$ as $h\rightarrow 0$, $h\mu\rightarrow -\infty$,
and $h\mu\in\mathcal{S}_I$. Here, $\mathcal{S}_I$ is the stability region of the implicit GLM \eqref{implicit}.
\end{theorem}

\begin{proof}
The result follows directly from \cite[Theorem 3]{zs13} on PR-convergence of IMEX-GLMs. \flushright{$\square$}
\end{proof}

\setcounter{equation}{0}
\setcounter{figure}{0}
\setcounter{table}{0}

\section{Stability analysis of IMEX GLMs} \label{sec:stability}
To analyze stability properties of IMEX GLMs (\ref{eq2.1}) we
use the test equation
\begin{equation} \label{eq4.1}
y'(t)=\lambda_0y(t)+\lambda_1y(t),
\quad
t\geq 0,
\end{equation}
where $\lambda_0$ and $\lambda_1$ are complex parameters.
Applying \eqref{extrap-imex} to (\ref{eq4.1}) and letting
$z_i=h\lambda_i$, $i=0,1$, we obtain
$$
\left\{
\begin{array}{c}
\big(\mbf{I}-(z_0\mbf{A}^{*}+z_1\mbf{A})\big)Y^{[n+1]}=z_0\bar{\mbf{A}}Y^{[n]}
+\mbf{U}y^{[n]}, \\ [3mm]
-(z_0\mbf{B}^{*}+z_1\mbf{B})Y^{[n+1]}+y^{[n+1]}=z_0\bar{\mbf{B}}Y^{[n]}
+\mbf{V}y^{[n]},
\end{array}
\right.
$$
$n=0,1,\ldots,N-1$. This is equivalent to the matrix recurrence relation
\begin{equation} \label{eq4.4}
\left[
\begin{array}{c}
Y^{[n+1]} \\
y^{[n+1]}
\end{array}
\right]=\mbf{M}(z_0,z_1)\left[
\begin{array}{c}
Y^{[n]} \\
y^{[n]}
\end{array}
\right],
\end{equation}
where the stability matrix $\mbf{M}(z_0,z_1)$ is defined by
$$
\mbf{M}(z_0,z_1)=\left[
\begin{array}{cc}
m_{11}(z_0,z_1)  & m_{12}(z_0,z_1) \\
m_{21}(z_0,z_1) & m_{22}(z_0,z_1)
\end{array}
\right]
$$
with
$$
m_{11}(z_0,z_1)=z_0\big(\mbf{I}-(z_0\mbf{A}^{*}+z_1\mbf{A})\big)^{-1}\bar{\mbf{A}},
$$
$$
m_{12}(z_0,z_1)=\big(\mbf{I}-(z_0\mbf{A}^{*}+z_1\mbf{A})\big)^{-1}\mbf{U},
$$
$$
m_{21}(z_0,z_1)=z_0\Big(\bar{\mbf{B}}+(z_0\mbf{B}^{*}+z_1\mbf{B})
\big(\mbf{I}-(z_0\mbf{A}^{*}+z_1\mbf{A})\big)^{-1}\bar{\mbf{A}}\Big),
$$
$$
m_{22}(z_0,z_1)=\mbf{V}+(z_0\mbf{B}^{*}+z_1\mbf{B})
\big(\mbf{I}-(z_0\mbf{A}^{*}+z_1\mbf{A})\big)^{-1}\mbf{U}.
$$
We define also the stability function of the IMEX GLM (\ref{eq2.1}) as a characteristic
polynomial of the stability matrix $\mbf{M}(z_0,z_1)$, i.e.,
$$
p(w,z_0,z_1)=\det\big(w\mbf{I}-\mbf{M}(z_0,z_1)\big).
$$
For $z_0=0$ the stability matrix $\mbf{M}(0,z_1)$ and polynomial $p(w,0,z_1)=w^s\eta(w,z_1)$
are those of the underlying GLM (\ref{eq1.2}).
For $z_1=0$ we obtain $\mbf{M}(z_0,0)$,
and it can be verified that this corresponds to the stability matrix
of the explicit method (\ref{explicit}).

We say that the IMEX GLM (\ref{eq2.1}) is stable for given $z_0,z_1\in\C$ if
all the roots $w_i(z_0,z_1)$, $i=1,2,\ldots,s+r$, of the stability function
$p(w,z_0,z_1)$ are inside of the unit circle.
As observed in \cite{hv03} in the context of IMEX $\theta$-methods
of order one,  a large region of absolute stability for the
explicit method (\ref{explicit})  and good stability properties
(for example $A$- or $L$-stability) for the
implicit method are not
sufficient to guarantee desirable stability
properties of the overall IMEX GLM (\ref{eq2.1}). We have to investigate the stability properties of the combined as IMEX GLM \cite{zhasan13,zs13}.

In this paper
we will be mainly interested in IMEX schemes which
are $A(\alpha)$- or $A$-stable with respect to the implicit part $z_1\in\C$.
To investigate such methods
we consider, similarly as in \cite{hv03,zs13}, the sets
\begin{equation} \label{eq4.5}
\mathcal{S}_{\alpha}=
\left\{ \begin{array}{lll}
z_0\in \C &:& \ \textrm{the IMEX GLM is stable for any }  \\
&&z_1\in \C: \ \mathcal{R}(z_1)<0 ~~ \textrm{and} ~~ \big|\Im(z_1)\big|\leq \tan(\alpha)\big|\mathcal{R}(z_1)\big|
\end{array} \right\}.
\end{equation}
%
For fixed values of $y\in \R$ we define also the sets
\begin{equation} \label{eq4.6}
\mathcal{S}_{\alpha,y}=
\left\{
\begin{array}{lcl}
z_0\in \C\!\!\!\!&:&  \textrm{the IMEX GLM is stable for fixed} \\
& & z_1=-|y|/\tan(\alpha)+iy
\end{array}
\right\}.
\end{equation}
It follows from the maximum principle that
\begin{equation} \label{eq4.7}
\mathcal{S}_{\alpha}=\bigcap_{y\in\R}\mathcal{S}_{\alpha,y}.
\end{equation}
Observe also that the region $\mathcal{S}_{\alpha,0}$ is independent of
$\alpha$, and corresponds to the region of absolute stability
of the explicit method (\ref{explicit}).
This region will be denoted by $\mathcal{S}_E$.
We have
\begin{equation} \label{eq4.8}
\mathcal{S}_{\alpha}\subset \mathcal{S}_E,
\end{equation}
and we will look for IMEX GLMs for which the stability region $\mathcal{S}_{\alpha}$ contains a
large part of the stability region $\mathcal{S}_E$ of the explicit method (\ref{explicit}),
for some $\alpha\in (0,\pi/2]$, preferably for $\alpha=\pi/2$.

The boundary $\partial \mathcal{S}_{\alpha,y}$ of the region $\mathcal{S}_{\alpha,y}$ can be determined by
the boundary locus method which computes the locus of the curve
$$
\partial \mathcal{S}_{\alpha,y}=\Big\{z_0\in\C: \ p\big(e^{i\theta},z_0,-|y|/\tan(\alpha)+iy\big)=0,
\ \theta\in[0,2k\pi]\Big\},
$$
where $k$ is a positive integer.

\begin{figure}[t!h!b!]
\begin{center}
\includegraphics[width=0.75\textwidth]{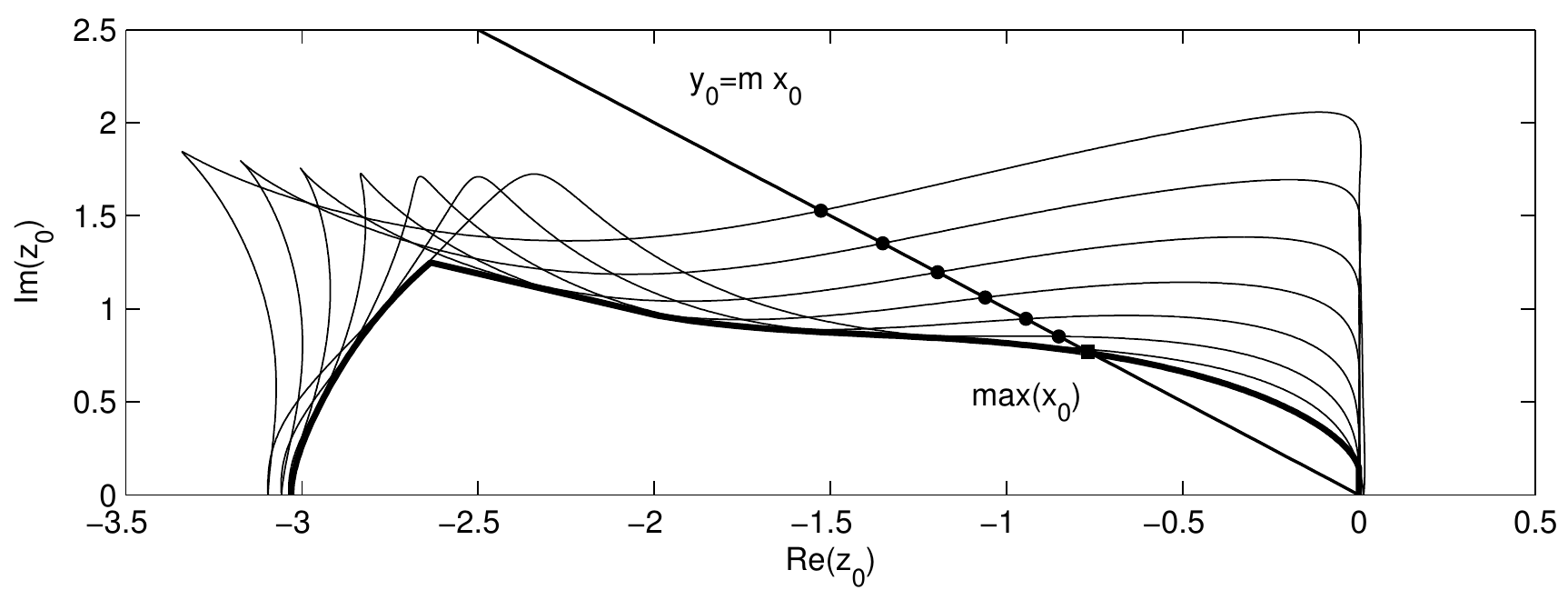}
\caption{Points on the intersection of the ray $y_0=mx_0$ and
$\partial\mathcal{S}_{\pi/2,y}$ for $y=-1.5,-1.0,\ldots,1.5$ (circles)
and on the intersection of $y_0=mx_0$ and $\partial\mathcal{S}_{\pi/2}$ (square).
This figure corresponds to IMEX GLM scheme with $p=q=r=s=2$ for $m=-1$, $\lambda=0.3$,
and $\beta_{21}=4.3$} \label{fig4.1}
\end{center}
\end{figure}

We have also developed an algorithm to determine the boundary $\partial \mathcal{S}_{\alpha}$
of the stability
region $\mathcal{S}_{\alpha}$. For fixed direction $m$ corresponding to the ray
$$
y_0=mx_0,
$$
and for fixed $z_1=-|y|/\tan(\alpha)+iy$ (or any $z_1\in\C$) we can compute the point
$z_0=x_0+iy_0$ of intersection of the boundary $\partial \mathcal{S}_{\alpha,y}$ of
$\mathcal{S}_{\alpha,y}$ with the ray $y_0=mx_0$ taking into account that such a point satisfies the
condition
$$
\max_{i=1,2,\ldots,s+r}\Big|w_i\big(z_0,-|y|/\tan(\alpha)+iy\big)\Big|=1.
$$
This can be done
by the bisection method which terminates when the
condition
\begin{equation} \label{eq4.9}
\Big|\max_{i=1,2,\ldots,s+r}\Big|w_i\big(z_0,-|y|/\tan(\alpha)+iy\big)\Big|-1\Big| \leq tol,
\end{equation}
with accuracy tolerance $tol$ is satisfied.
We apply this method to the interval $[\bar{x}_0,0]$ with $\bar{x}_0$ large enough so
that the condition (\ref{eq4.9}) is not satisfied for the first iteration of bisection method.
This process leads to the definition of the function
$$
x_0=f(m,\alpha,y),
$$
where $x_0$ corresponds to the points $z_0=x_0+iy_0\in \partial\mathcal{S}_{\alpha,y}$.
Then the boundary $\partial\mathcal{S}_{\alpha}$ of the region $\mathcal{S}_{\alpha}$
can be determined by minimizing the negative value of this function for
$m\in\R$ and plotting the resulting points $z_0=x_0+iy_0$.
This algorithm is illustrated in Fig.~\ref{fig4.1},
where we have plotted
points on the intersection of the ray $y_0=mx_0$ and
$\partial\mathcal{S}_{\pi/2,y}$ for $y=-1.5,-1.0,\ldots,1.5$ (circles)
and on the intersection of $y_0=mx_0$ and $\partial\mathcal{S}_{\pi/2}$ (square).
This figure corresponds to IMEX GLM scheme with $p=q=r=s=2$ for $m=-1$, $\lambda=0.3$,
and $\beta_{21}=4.3$.
This minimization can be
accomplished using the subroutine \texttt{fminsearch.m} in Matlab
applied to $f(m,\alpha,y)$ for fixed values of $m\in\R$ and $\alpha\in[0,\pi/2]$
starting with
appropriately chosen initial guesses for $y$.
This process will be next applied to specific IMEX GLMs.

\setcounter{equation}{0}
\setcounter{figure}{0}
\setcounter{table}{0}

\section{Construction of IMEX GLMs with desirable stability properties} \label{sec:construction}
In this section we describe the construction of IMEX GLMs (\ref{eq2.1}) up to the
order $p=4$ with large regions of absolute stability $\mathcal{S}$ with respect to the explicit part
assuming that the implicit part is $A$- or $L$-stable.

We will always start with implicit DIMSIM with $p=q=r=s$ and the abscissa
vector $\mbf{c}$ given in advance. After computing the coefficient
matrices $\mbf{A}$ and $\mbf{V}$ so that the resulting method has Runge-Kutta
stability with the underlying Runge-Kutta formula which is
$A$- or $L$-stable, the coefficient matrix $\mbf{B}$ is computed from the relation
\begin{equation} \label{eq5.1}
\mbf{B}=\mbf{B}_0-\mbf{A}\mbf{B}_1-\mbf{V}\mbf{B}_2+\mbf{V}\mbf{A}.
\end{equation}
Here, $\mbf{B}_0$, $\mbf{B}_1$, and $\mbf{B}_2$ are $s\times s$ matrices with
the $(i,j)$ elements
given by
$$
\ds\frac{\ds\int_0^{1+c_i}\phi_j(x)dx}{\phi_j(c_j)},
\quad
\ds\frac{\phi_j(1+c_i)}{\phi_j(c_j)},
\quad
\ds\frac{\ds\int_0^{c_i}\phi_j(x)dx}{\phi_j(c_j)},\quad
\phi_i(x)=\ds\prod_{j=1,j\neq i}^s(x-c_j),
$$
$i=1,2,\ldots,s$, compare Th.~$5.1$ in \cite{but93} or Th.~$3.2.1$ in \cite{jac09}.

\subsection{IMEX GLMs with $p=q=r=s=1$} \label{sec5.1}
Consider the implicit $\theta$-method defined by
\begin{equation} \label{eq5.2}
\left\{
\begin{array}{l}
Y^{[n+1]}=h\theta\big(f(Y^{[n+1]})+g(Y^{[n+1]})\big)+y^{[n]}, \\ [3mm]
y^{[n+1]}=h\big(f(Y^{[n+1]})+g(Y^{[n+1]})\big)+y^{[n]},
\end{array}
\right.
\end{equation}
$n=0,1,\ldots,N-1$.
This method is $A$-stable for $\theta\in[1/2,1]$ and $L$-stable for
$\theta\in(1/2,1]$. Consider also the extrapolation procedure
\begin{equation} \label{eq5.3}
f(Y^{[n+1]})=f(Y^{[n]}),
\end{equation}
$n=1,2,\ldots,N-1$.
Substituting (\ref{eq5.3}) into (\ref{eq5.2}) we obtain IMEX $\theta$-method of the form
\begin{equation} \label{eq5.4}
\left\{
\begin{array}{l}
Y^{[n+1]}=h\theta\big(f(Y^{[n]})+g(Y^{[n+1]})\big)+y^{[n]}, \\ [3mm]
y^{[n+1]}=h\big(f(Y^{[n]})+g(Y^{[n+1]})\big)+y^{[n]},
\end{array}
\right.
\end{equation}
$n=1,2,\ldots,N-1$.
We would like to point out that this variant of IMEX scheme is different
from IMEX $\theta$-method considered in \cite{hv03}.
Observe that the method (\ref{eq5.4}) requires a starting procedure
to compute $Y^{[1]}\approx y(t_0+\theta h)$ and $y^{[1]}\approx y(t_1)$.

The method (\ref{eq5.2}) can be represented by the abscissa $\mbf{c}=\theta$, the partitioned
matrix
\begin{equation*}
\left[
\begin{array}{c|c}
\mbf{A} & \mbf{U} \\
\hline
\mbf{B} & \mbf{V}
\end{array}
\right]=\left[
\begin{array}{c|c}
\theta \ & \ 1 \\
\hline
 1 \ & \ 1
\end{array}
\right],
\end{equation*}
and $\mbf{q}_0=1$, $\mbf{q}_1=0$,
and the explicit formula corresponding to $g(y)=0$ in (\ref{eq5.4}) is GLM
with $\mbf{c}=[\theta-1,\theta]^T$,
\begin{equation} \label{eq5.5}
\left[
\begin{array}{c|c}
\mbf{A} & \mbf{U} \\
\hline
\mbf{B} & \mbf{V}
\end{array}
\right]=\left[
\begin{array}{cc|cc}
0 & 0  & \ 1 & 0 \\
\theta & 0  & \ 0 & 1 \\
\hline
\theta & 0  & \ 0 & 1 \\
1 & 0   & \ 0 & 1
\end{array}
\right],
\end{equation}
and $\mbf{q}_0=[1,1]^T$, $\mbf{q}_1=[\theta-1,0]^T$.

\begin{figure}[t!h!b!]
\begin{center}
\includegraphics[width=0.75\textwidth]{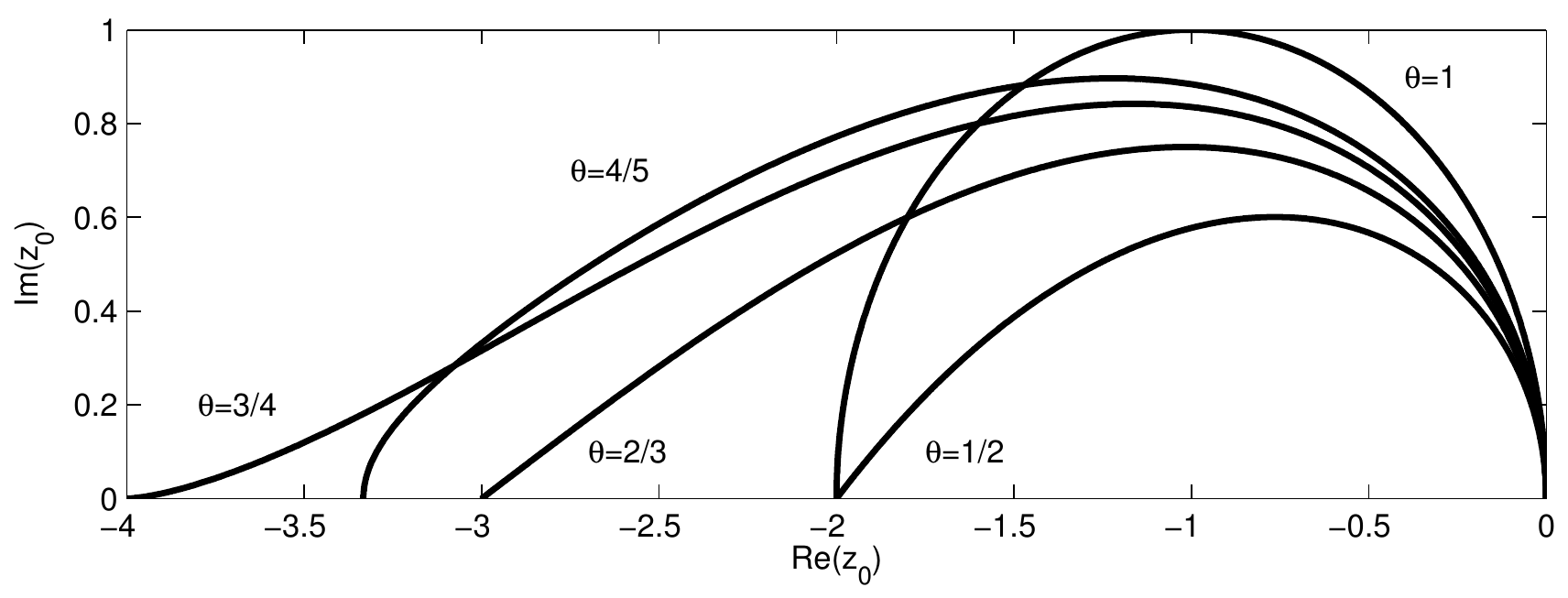}
\caption{Stability regions $\mathcal{S}_E=\mathcal{S}_E(\theta)$ of
explicit methods for $\theta=1/2$, $2/3$, $3/4$, $4/5$, and $1$} \label{fig5.1}
\end{center}
\end{figure}
\begin{figure}[t!h!b!]
\begin{center}
\includegraphics[width=0.75\textwidth]{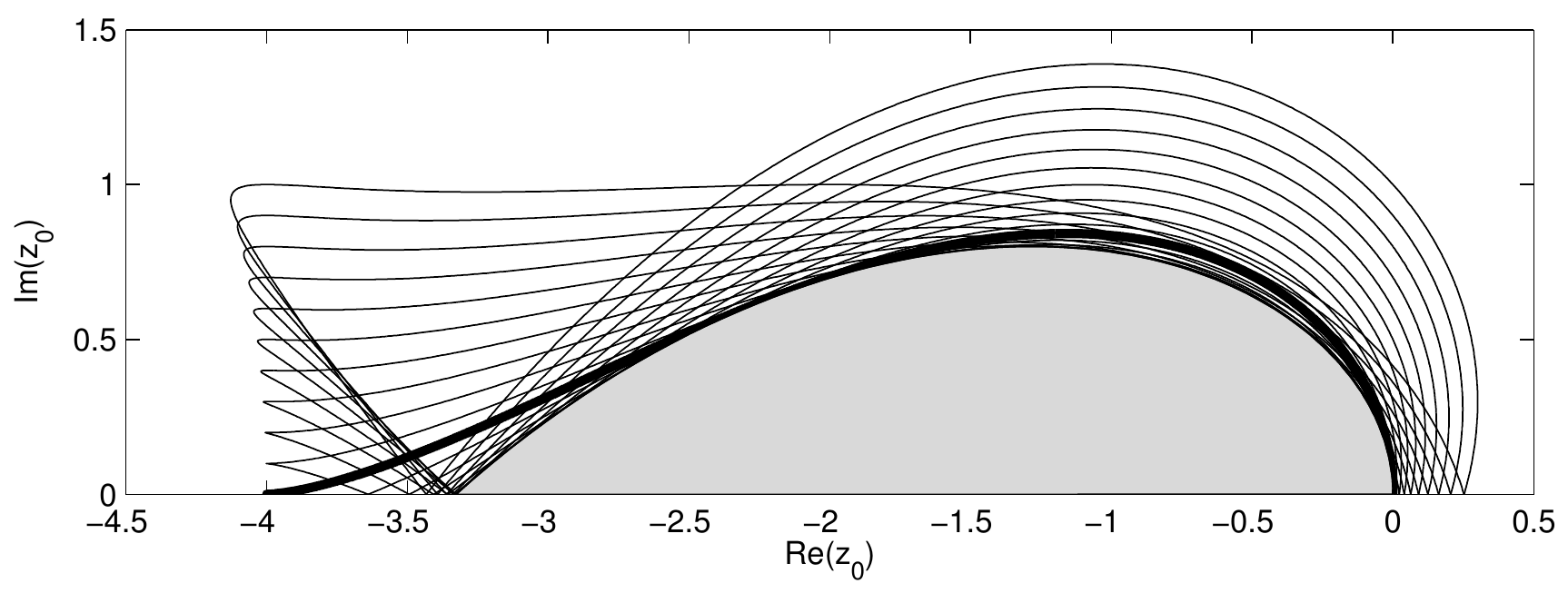}
\caption{Stability regions $\mathcal{S}_{\pi/2,y}(\theta)$, $y=-1.0,-0.9,\ldots,1.0$ (thin lines),
$\mathcal{S}_{\pi/2}(\theta)$ (shaded region), and $\mathcal{S}_E(\theta)$ (thick line)
for $\theta=3/4$} \label{fig5.2}
\end{center}
\end{figure}
\begin{figure}[t!h!b!]
\begin{center}
\includegraphics[width=0.75\textwidth]{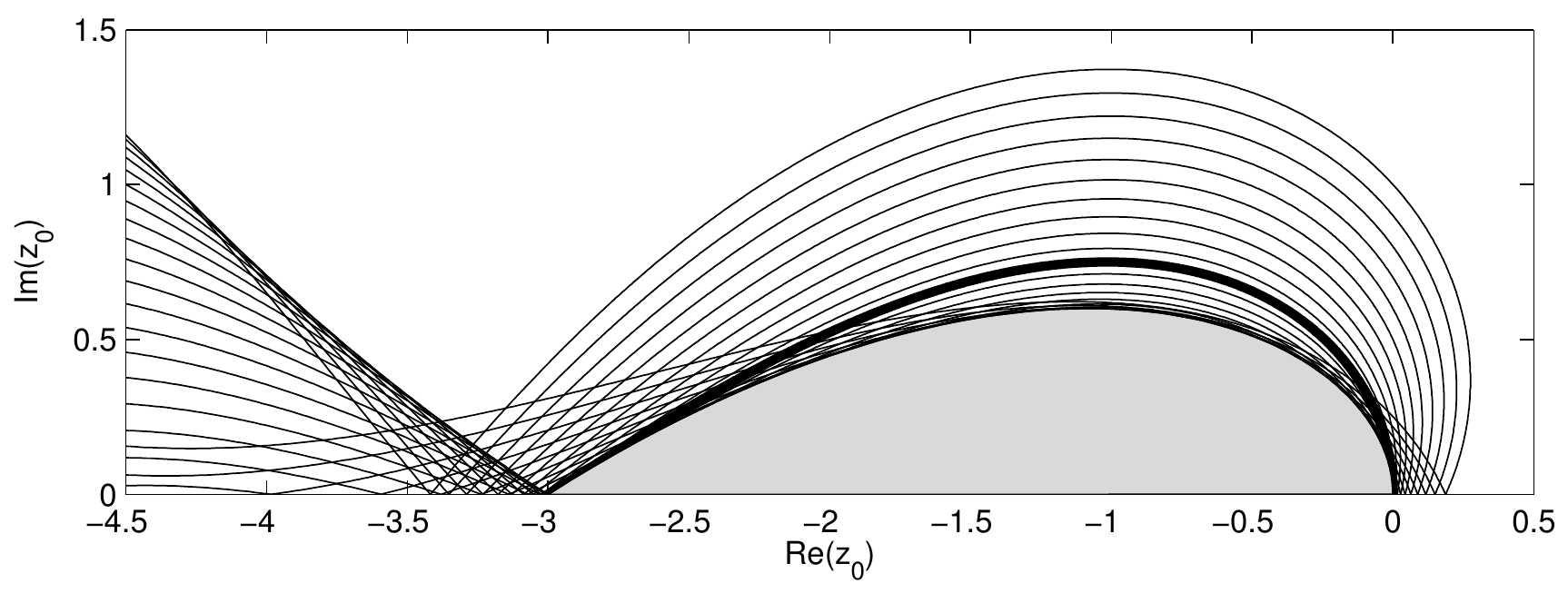}
\caption{Stability regions $\mathcal{S}_{\pi/2,y}(\theta)$, $y=-1.0,-0.9,\ldots,1.0$ (thin lines),
$\mathcal{S}_{\pi/2}(\theta)$ (shaded region), and $\mathcal{S}_E(\theta)$ (thick line)
for $\theta=2/3$} \label{fig5.3}
\end{center}
\end{figure}

It can be verified that the stability matrix of IMEX scheme (\ref{eq5.4}) takes the form
$$
\mbf{M}(z_0,z_1)=
\ds\frac{1}{1-\theta z_1}
\left[
\begin{array}{cc}
\theta z_0 & 1 \\
z_0 & 1+(1-\theta)z_1
\end{array}
\right].
$$
To investigate stability properties of (\ref{eq5.4}) it is
more convenient to work with the polynomial obtained by multiplying the
characteristic function $p(w,z_0,z_1)$ of $\mbf{M}(z_0,z_1)$ by a factor
$(1-\theta z_1)^2$. The resulting quadratic polynomial,
which will be denoted by the same symbol
$p(w,z_0,z_1)$, takes the form
$$
p(w,z_0,z_1)=(1-\theta z_1)^2w^2
-(1-\theta z_1)\big(1+\theta z_0+(1-\theta)z_1\big)w
-(1-\theta)z_0(1-\theta z_1).
$$
The stability polynomial of the explicit methods (\ref{eq5.5})
obtained by letting $g(y)=0$ in (\ref{eq5.4})
corresponds to $z_1=0$ and is given by
$$
p(w,z_0,0)=w^2-(1+\theta z_0)w-(1-\theta)z_0.
$$
The stability regions $\mathcal{S}_E=\mathcal{S}_E(\theta)$
corresponding to this polynomial are plotted in Fig~\ref{fig5.1} for
$\theta=1/2$, $2/3$, $3/4$, $4/5$ and $1$.
Observe that the stability region of the method (\ref{eq5.5}) corresponding
to $\theta=1$ is the unit disk
$$
\mathcal{S}_E=\mathcal{S}_E(1)=\big\{z_0\in\C: \ |z_0+1|<1\big\}.
$$

We will investigate next the regions $\mathcal{S}_{\pi/2}=\mathcal{S}_{\pi/2}(\theta)$
defined by (\ref{eq4.5}).
We analyze first the case $\theta=1$. It follows from Schur criterion applied
to the polynomial $p(w,z_0,iy)$ that $z_0\in\mathcal{S}_{\pi/2}(1)$ if
and only if
$$
y^2-2x_0-x_0^2-y_0^2>0
$$
for any $y\in\R$.
This is equivalent to
$$
(x_0+1)^2+y_0^2<1
\quad
\textrm{or}
\quad
|z_0+1|<1
$$
and it follows that in this case $\mathcal{S}_{\pi/2}(1)=\mathcal{S}_E(1)$.
For $\theta\in(1/2,1)$ $\mathcal{S}_{\pi/2}(\theta)$ is no longer equal to $\mathcal{S}_E(\theta)$,
but $\mathcal{S}_{\pi/2}(\theta)$ contains some part of $\mathcal{S}_E(\theta)$. This is
illustrated in Fig.~\ref{fig5.2} for $\theta=3/4$ and in Fig.~\ref{fig5.3}
for $\theta=2/3$.
In these figures we have plotted the regions $\mathcal{S}_{\pi/2,y}=\mathcal{S}_{\pi/2,y}(\theta)$
defined by (\ref{eq4.6}) for
$y=-1.0,-0.9,\ldots,1.0$ (thin lines), the regions $\mathcal{S}_{\pi/2}=\mathcal{S}_{\pi/2}(\theta)$
defined by (\ref{eq4.5}) (shaded regions), and the regions
$\mathcal{S}_E=\mathcal{S}_E(\theta)$ corresponding to explicit methods
(\ref{eq5.5}) for $\theta=3/4$ and $\theta=2/3$. We can see that for these values of
$\theta$ the sets $\mathcal{S}_{\pi/2}(\theta)$ contain quite large parts
of $\mathcal{S}_E(\theta)$.
These figures illustrate also
the relation (\ref{eq4.7}) and the inclusion (\ref{eq4.8}).
For $\theta=1/2$ it follows from Schur criterion that
$z_0\in\mathcal{S}_{\pi/2}(1/2)$ if and only if
$$
x_0^2+y_0^2<4
\quad
\textrm{and}
\quad
x_0y^2-4yy_0-x_0(x_0+2)^2-y_0^2(4+x_0)>0
$$
for any $y\in\R$.
This implies that
$$
x_0>0
\quad
\textrm{and}
\quad
(2+x_0)^2(x_0^2+y_0^2)<0
$$
and it follows that $\mathcal{S}_{\pi/2}(1/2)$ is empty.

\subsection{IMEX GLMs with $p=q=r=s=2$} \label{sec5.2}
Consider the implicit DIMSIM with $\mbf{c}=[0,1]^T$, the coefficient
matrices given by
$$
\left[
\begin{array}{c|c}
\mbf{A} & \mbf{U} \\
\hline
\mbf{B} & \mbf{V}
\end{array}
\right]=\left[
\begin{array}{cc|cc}
\lambda & 0 & 1 & 0 \\
\frac{2}{1+2\lambda} & \lambda & 0 & 1 \\
\hline
\frac{8\lambda^3+12\lambda^2-2\lambda+5}{4(2\lambda+1)} & \frac{1-4\lambda^2}{4} &
\frac{1}{2}+\lambda & \frac{1}{2}-\lambda \\
\frac{8\lambda^3+20\lambda^2-2\lambda+3}{4(2\lambda+1)} &
\frac{-8\lambda^3-12\lambda^2+10\lambda-1}{4(2\lambda+1)} &
\frac{1}{2}+\lambda & \frac{1}{2}-\lambda
\end{array}
\right],
$$
and the vectors $\mbf{q}_0$, $\mbf{q}_1$, and $\mbf{q}_2$ equal to
$$
\mbf{q}_0=\left[
\begin{array}{c}
1 \\
1
\end{array}
\right],
\quad
\mbf{q}_1=\left[
\begin{array}{c}
-\lambda \\
\frac{-2\lambda^2+\lambda-1}{2\lambda+1}
\end{array}
\right],
\quad
\mbf{q}_2=\left[
\begin{array}{c}
0 \\
\frac{1-2\lambda}{2}
\end{array}
\right].
$$
It was demonstrated in \cite{jac09} that this method has order $p=2$ and
stage order $q=2$. Moreover, this method is $A$-stable if $\lambda\geq 1/4$
and $L$-stable for $\lambda=(2\pm \sqrt{2})/2$.

The coefficients $\alpha_{jk}$ of the extrapolation formula (\ref{eq1.11})
of order $p=2$ computed from the system (\ref{eq2.12}) corresponding to
$p=s=2$ take the form
$$
\alpha=\left[
\begin{array}{cc}
\alpha_{11} & \alpha_{12} \\
\alpha_{21} & \alpha_{22}
\end{array}
\right]=\left[
\begin{array}{cc}
\phantom{-}0 & 1 \\
-1 & 2-\beta_{21}
\end{array}
\right],
$$
and the matrices $\bar{\mbf{A}}$, $\mbf{A}^{*}$,  $\bar{\mbf{B}}$, and $\mbf{B}^{*}$
appearing in the representation of IMEX GLM (\ref{eq2.1}), and the corresponding
explicit scheme (\ref{extrap-imex}) or (\ref{explicit}), are given by
$$
\bar{\mbf{A}}=\left[
\begin{array}{cc}
 0 & \lambda \\
 -\lambda & \frac{2+(2-\beta_{21})\lambda+2(2-\beta_{21})\lambda^2}{1+2\lambda}
\end{array}
\right],
\quad
\mbf{A}^{*}=\left[
\begin{array}{cc}
 0 & 0 \\
 \beta_{21}\lambda & 0
\end{array}
\right],
$$
$$
\bar{\mbf{B}}=\left[
\begin{array}{cc}
\frac{4\lambda^2-1}{4} &
\frac{7-\beta_{21}+2(1-\beta_{21})\lambda+4(1+\beta_{21})\lambda^2-8(1-\beta_{21})\lambda^3}{4(1+2\lambda)} \\
\frac{1-10\lambda+12\lambda^2+8\lambda^3}{4(1+2\lambda)} &
\frac{1+\beta_{21}+2(9-5\beta_{21})\lambda-4(1-3\beta_{21})\lambda^2-8(1-\beta_{21})\lambda^3}{4(1+2\lambda)}
\end{array}
\right],
$$
$$
\mbf{B}^{*}=\left[
\begin{array}{cc}
\frac{\beta_{21}(1-4\lambda^2)}{4} & 0 \\
\frac{-\beta_{21}(1-10\lambda+12\lambda^2+8\lambda^3)}{4(1+2\lambda)} & 0
\end{array}
\right].
$$

\begin{figure}[t!h!b!]
\begin{center}
\includegraphics[width=0.75\textwidth]{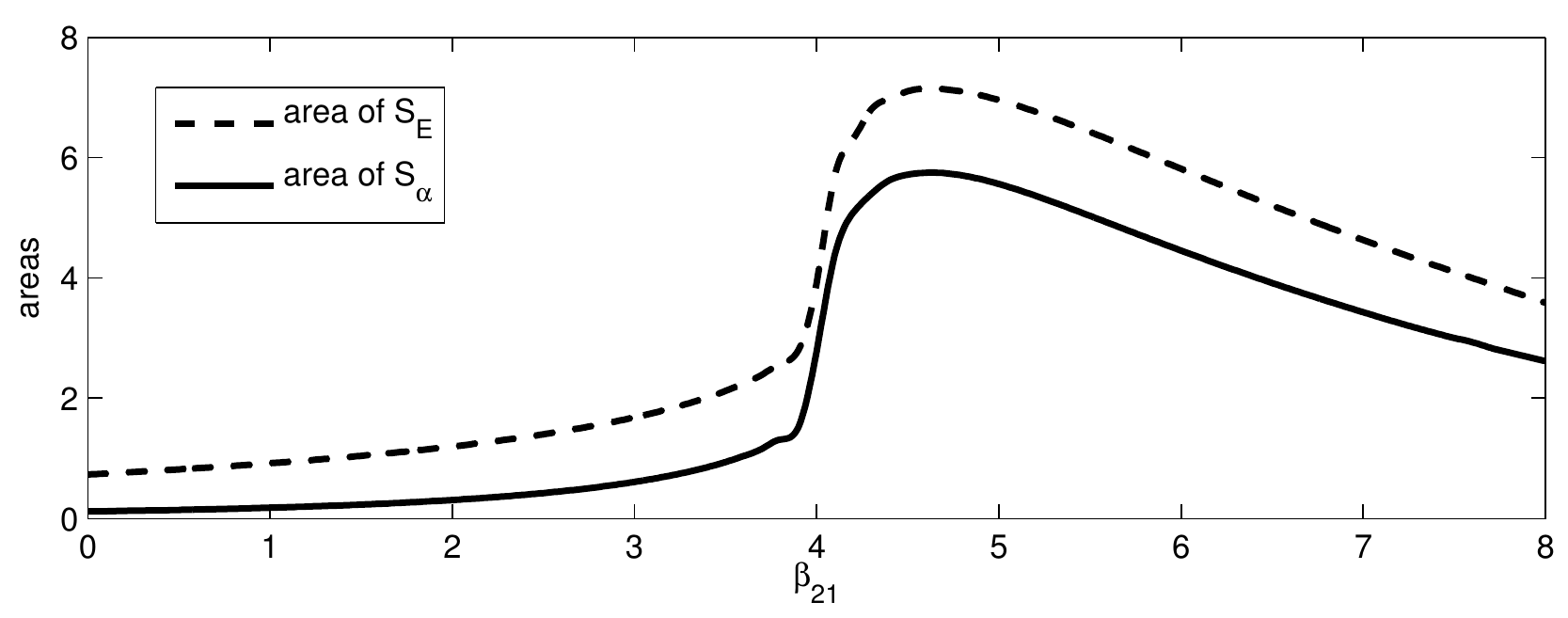}
\caption{Areas of stability regions $\mathcal{S}_E=\mathcal{S}_E(\beta_{21})$ and
$\mathcal{S}_{\alpha}=\mathcal{S}_{\alpha}(\beta_{21})$ for
$\alpha=\pi/2$ and $\beta_{21}\in [0,8]$} \label{fig5.4}
\end{center}
\end{figure}
\begin{figure}[t!h!b!]
\begin{center}
\includegraphics[width=0.75\textwidth]{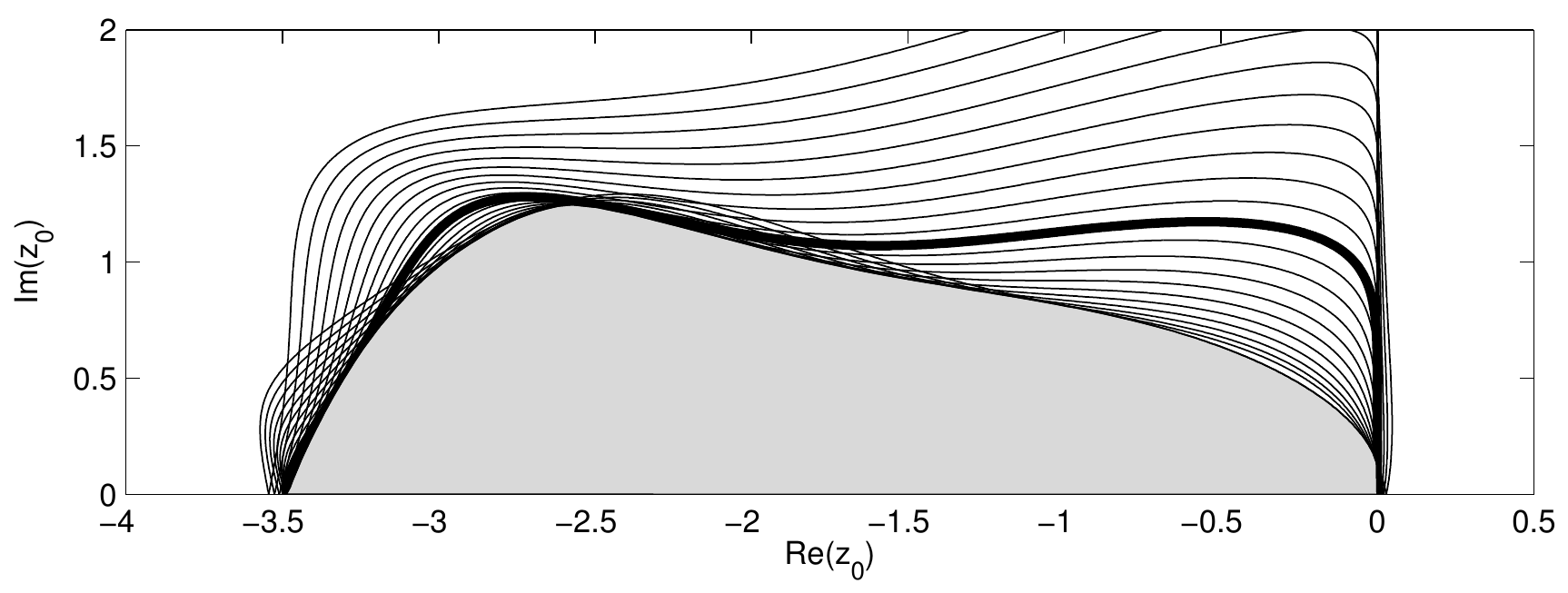}
\caption{Stability regions $\mathcal{S}_{\pi/2,y}$, $y=-2.0,-1.8,\ldots,2.0$ (thin lines),
$\mathcal{S}_{\pi/2}$ (shaded region), and $\mathcal{S}_E$ (thick line)
for $\lambda=(2-\sqrt{2})/2$ and $\beta_{21}\approx 4.64$} \label{fig5.5}
\end{center}
\end{figure}
\begin{figure}[t!h!b!]
\begin{center}
\includegraphics[width=0.75\textwidth]{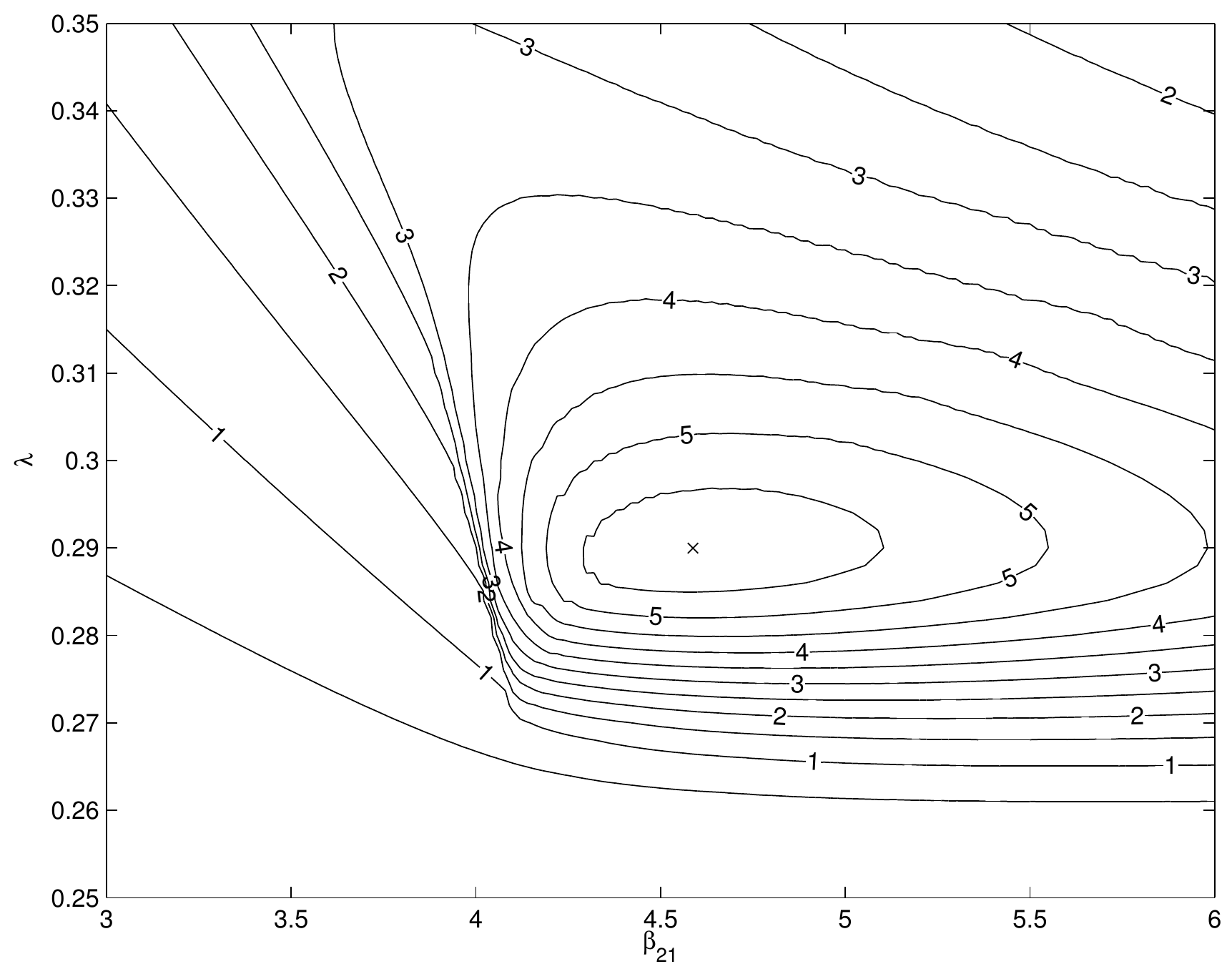}
\caption{Contour plots of the area of stability region
$\mathcal{S}_{\pi/2}$ of IMEX GLMs for $p=q=r=s=2$} \label{fig5.6}
\end{center}
\end{figure}
\begin{figure}[t!h!b!]
\begin{center}
\includegraphics[width=0.75\textwidth]{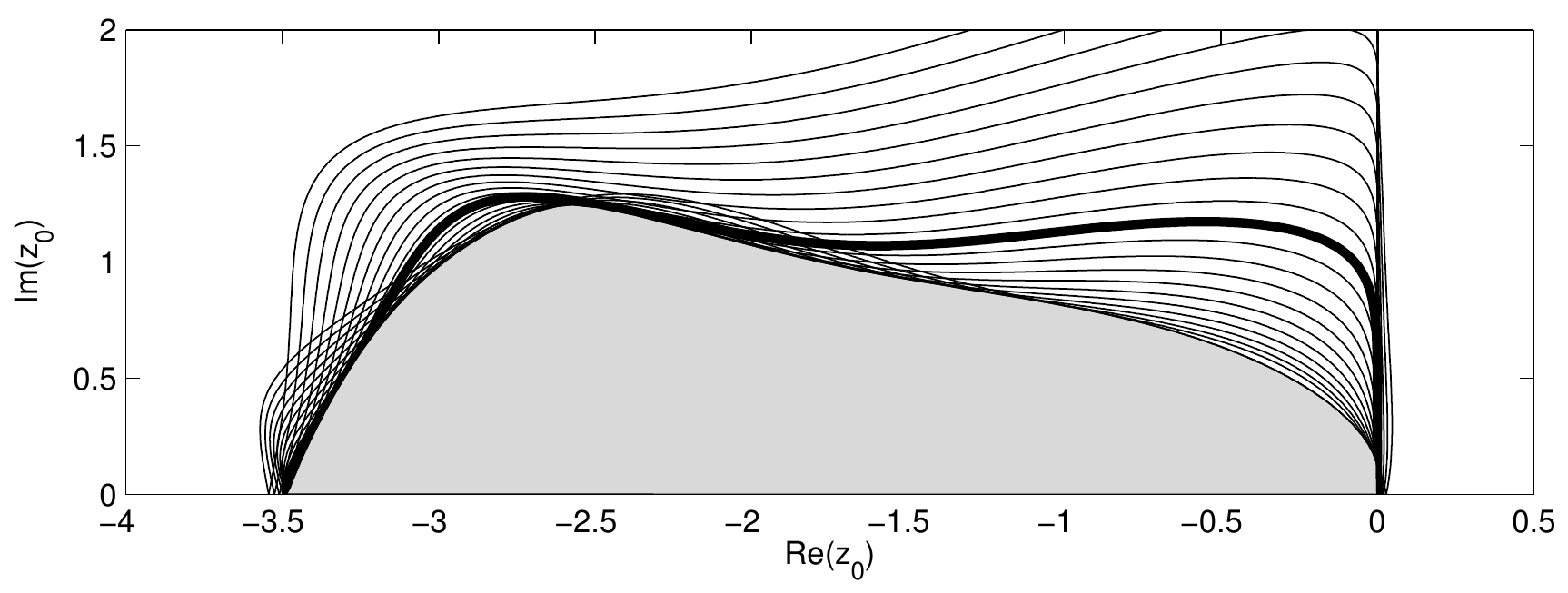}
\caption{Stability regions $\mathcal{S}_{\pi/2,y}$, $y=-2.0,-1.8,\ldots,2.0$ (thin lines),
$\mathcal{S}_{\pi/2}$ (shaded region), and $\mathcal{S}_E$ (thick line)
for $\lambda \approx 0.29$ and $\beta_{21}\approx 4.64$} \label{fig5.7}
\end{center}
\end{figure}

To investigate the stability properties of the resulting IMEX scheme
we will work with the stability polynomial
$p(w,z_0,z_1)$ obtained by multiplying stability function of
the method by a factor $(1-\lambda z_1)^2$. It can be verified that
this polynomial takes the form
$$
p(w,z_0,z_1)=
w\big((1-\lambda z_1)^2w^3-p_3(z_0,z_1)w^2
+p_2(z_0,z_1)w-p_1(z_0,z_1)\big),
$$
with the coefficients $p_1(z_0,z_1)$, $p_2(z_0,z_1)$, and $p_3(z_0,z_1)$ which
depend also on $\lambda$ and $\beta_{21}$. These coefficients are given by
$$
\begin{array}{lcl}
p_1(z_0,z_1) \ = \ \ds\frac{4-\beta_{21}-2(4-\beta_{21})\lambda+4\beta_{21}\lambda^2-8\beta_{21}\lambda^3}
{4(1+2\lambda)}z_0
+\ds\frac{1-4\lambda+2\lambda^2}{2}z_0^2,
\end{array}
$$
$$
\begin{array}{lcl}
p_2(z_0,z_1) & = & \ds\frac{2(1-\lambda+(2-\beta_{21})\lambda^2-2\beta_{21}\lambda^3)}
{1+2\lambda}z_0
+\ds\frac{\beta_{21}-4\beta_{21}\lambda+2(1+\beta_{21})\lambda^2}
{2}z_0^2 \\ [5mm]
& - &
\ds\frac{2-\beta_{21}-4(2-\beta_{21})\lambda+2(2-\beta_{21})\lambda^2}{2}z_0z_1,
\end{array}
$$
$$
\begin{array}{lcl}
p_3(z_0,z_1) & = &
1+\ds\frac{8+\beta_{21}+2(4-\beta_{21})\lambda+4(4-3\beta_{21})\lambda^2-8\beta_{21}\lambda^3}
{4(1+2\lambda)}z_0  \\ [5mm]
& + &
\beta_{21}\lambda^2z_0^2-(2-\beta_{21})z_0z_1+(1-2\lambda)z_1
+\ds\frac{1-4\lambda+2\lambda^2}{2}z_1^2.
\end{array}
$$

The underlying implicit GLM is $A$- and $L$-stable for
$\lambda=(2\pm\sqrt{2})/2$, compare \cite{jac09}, and we choose $\lambda=(2-\sqrt{2})/2$
since this value leads to explicit methods and IMEX schemes with larger regions of stability
$ \mathcal{S}_E$ and $\mathcal{S}_{\pi/2}$ than those corresponding to $\lambda=(2+\sqrt{2})/2$.
We have plotted in Fig~\ref{fig5.4} the area of the stability region
$\mathcal{S}_E=\mathcal{S}_E(\beta_{21})$
of the explicit method (corresponding to $z_1=0$)
and the area of the stability region $\mathcal{S}_{\pi/2}=\mathcal{S}_{\pi/2}(\beta_{21})$
of the IMEX scheme for $\beta_{21}\in[0,8]$.
It can be verified that the explicit formula attains the maximal area of
$\mathcal{S}_E$, approximately
equal to $7.15$ for $\beta_{21}\approx 4.56$, and the IMEX scheme attains the
maximal area of $\mathcal{S}_{\pi/2}$, approximately equal to $5.75$ for
$\beta_{21}\approx 4.64$.

On Fig~\ref{fig5.5} we have plotted stability regions $\mathcal{S}_{\pi/2,y}$ for
$y=-2.0,-1.8,\ldots,2.0$ (thin lines), stability region $\mathcal{S}_{\pi/2}$ (shaded region),
and stability region $\mathcal{S}_E$ (thick line). We can see that $\mathcal{S}_{\pi/2}$ contains
a significant part of $\mathcal{S}_E$.

We have also displayed on Fig.~\ref{fig5.6} contour plots of the area of stability region
$\mathcal{S}_{\pi/2}$ of IMEX methods for $\beta_{21}\in[3,6]$ and $\lambda\in[0.25,0.35]$.
This area attains its maximum value approximately equal to $5.83$ for
$\beta_{21}\approx 4.59$ and $\lambda\approx 0.29$.
This point is marked by the symbol `$\times$' on Fig.~\ref{fig5.6}.
On Fig~\ref{fig5.7} we have plotted stability regions $\mathcal{S}_{\pi/2,y}$ for
$y=-2.0,-1.8,\ldots,2.0$ (thin lines), stability region $\mathcal{S}_{\pi/2}$ (shaded region),
and stability region $\mathcal{S}_E$ (thick line) corresponding to these
values of $\beta_{21}$ and $\lambda$. We can see again that $\mathcal{S}_{\pi/2}$ contains
a significant part of $\mathcal{S}_E$.

\subsection{IMEX GLMs with $p=q=r=s=3$} \label{sec5.3}

Let $\lambda\approx 0.43586652$ be a root of the cubic polynomial
$$
\varphi(\lambda)=\lambda^3-3\lambda^2+\ds\frac{3}{2}\lambda-\ds\frac{1}{6},
$$
and consider the implicit DIMSIM with $\mbf{c}=[0,1/2,1]^T$, the coefficient
matrix $\mbf{A}$ given by
$$
\mbf{A}=\left[
\begin{array}{ccc}
\phantom{-}0.43586652 & 0 & 0 \\
\phantom{-}0.25051488 & 0.43586652 & 0 \\
-1.2115943 & 1.0012746 & 0.43586652
\end{array}
\right],
$$
the rank one coefficient matrix $\mbf{V}=\mbf{e}\mbf{v}^T$,
where $\mbf{e}=[1,1,1]^T$ and
$$
\mbf{v}=\left[
\begin{array}{ccc}
0.55209096 & 0.73485666 & -0.28694762
\end{array}
\right]^T,
$$
and the vectors $\mbf{q}_0$, $\mbf{q}_1$, $\mbf{q}_2$, and $\mbf{q}_3$
equal to $\mbf{q}_0=\mbf{e}$,
$$
\mbf{q}_1=\left[
\begin{array}{ccc}
-0.43586652 &
-0.18638140 &
0.77445315
\end{array}
\right]^T,
$$
$$
\mbf{q}_2=\left[
\begin{array}{ccc}
0 &
-0.092933261&
-0.43650382
\end{array}
\right]^T,
$$
$$
\mbf{q}_3=\left[
\begin{array}{ccc}
0 &
-0.033649982 &
-0.17642592
\end{array}
\right]^T.
$$
Computing the coefficient matrix $\mbf{B}$ from the relation
(\ref{eq5.1}) leads to the method of order $p=3$ and stage order $q=3$.
It was demonstrated in \cite{jac09} that the resulting method is
$A$- and $L$-stable.

The coefficients $\alpha_{jk}$ of the extrapolation formula (\ref{eq1.11})
of order $p=3$ computed from the system (\ref{eq2.12}) corresponding to
$p=s=3$ take the form
$$
\alpha=\left[
\begin{array}{ccc}
\alpha_{11} & \alpha_{12} & \alpha_{13} \\
\alpha_{21} & \alpha_{22} & \alpha_{23} \\
\alpha_{31} & \alpha_{32} & \alpha_{33}
\end{array}
\right]=\left[
\begin{array}{ccc}
0 & \phantom{-}0 & 1 \\
1 & -3 & 3-\beta_{21} \\
3-\beta_{32} & 3\beta_{32}-8 & 6-\beta_{31}-3\beta_{32}
\end{array}
\right].
$$

To investigate the stability properties of the resulting IMEX scheme
we will work with the stability polynomial
$p(w,z_0,z_1)$ obtained by multiplying stability function of
the method by a factor $(1-\lambda z_1)^3$,
where $\lambda$ is the diagonal element of the matrix $\mbf{A}$.
It can be verified that
this polynomial takes the form
$$
\begin{array}{lcl}
p(w,z_0,z_1) & = &
(1-\lambda z_1)^3w^6-p_5(z_0,z_1)w^5
+p_4(z_0,z_1)w^4-p_3(z_0,z_1)w^3 \\ [2mm]
& + &
p_2(z_0,z_1)w^2-p_1(z_0,z_1)w+p_0(z_0,z_1),
\end{array}
$$
with the coefficients $p_0(z_0,z_1)$, $p_1(z_0,z_1)$, $p_2(z_0,z_1)$,
$p_3(z_0,z_1)$, $p_4(z_0,z_1)$, and $p_5(z_0,z_1)$ which
are polynomials of degree less than or equal to $3$ with respect to
$z_0$ and $z_1$. These coefficients
depend also on $\beta_{21}$, $\beta_{31}$, and $\beta_{32}$.

\begin{figure}[t!h!b!]
\begin{center}
\includegraphics[width=0.75\textwidth]{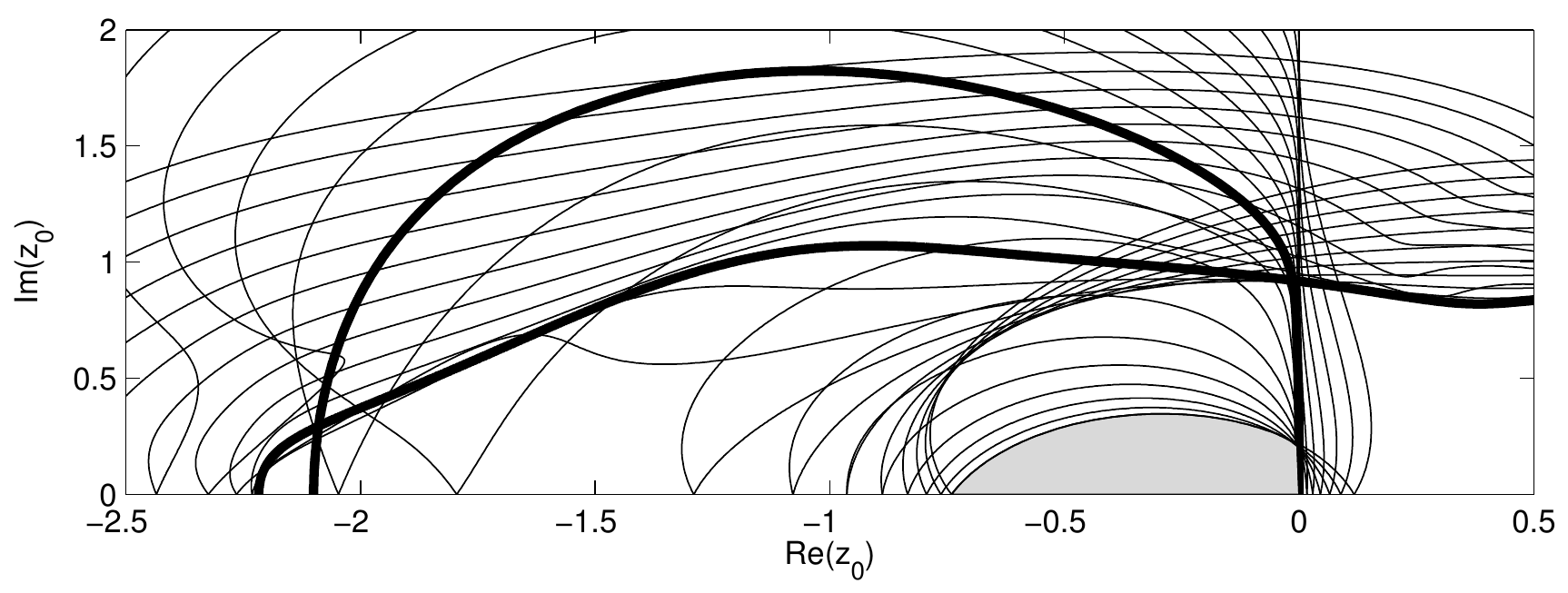}
\caption{Stability regions $\mathcal{S}_{\pi/2,y}$, $y=-2.0,-1.8,\ldots,2.0$ (thin lines),
$\mathcal{S}_{\pi/2}$ (shaded region), and $\mathcal{S}_E$ (thick line)
for $\beta_{21}\approx 1.13$, $\beta_{31}\approx 1.45$, and $\beta_{32}\approx -0.158$} \label{fig5.8}
\end{center}
\end{figure}
\begin{figure}[t!h!b!]
\begin{center}
\includegraphics[width=0.75\textwidth]{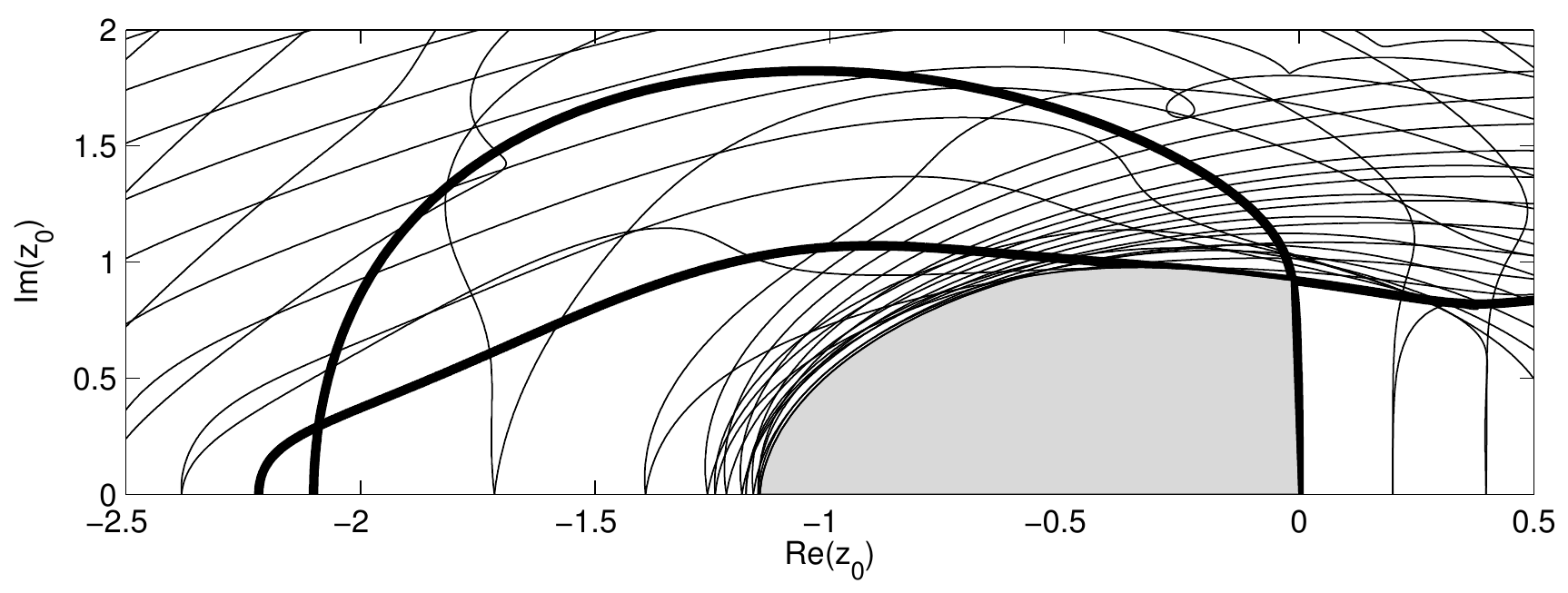}
\caption{Stability regions $\mathcal{S}_{\pi/4,y}$, $y=-2.0,-1.8,\ldots,2.0$ (thin lines),
$\mathcal{S}_{\pi/4}$ (shaded region), and $\mathcal{S}_E$ (thick line)
for $\beta_{21}\approx 1.13$, $\beta_{31}\approx 1.45$, and $\beta_{32}\approx -0.158$} \label{fig5.9}
\end{center}
\end{figure}
\begin{figure}[t!h!b!]
\begin{center}
\includegraphics[width=0.75\textwidth]{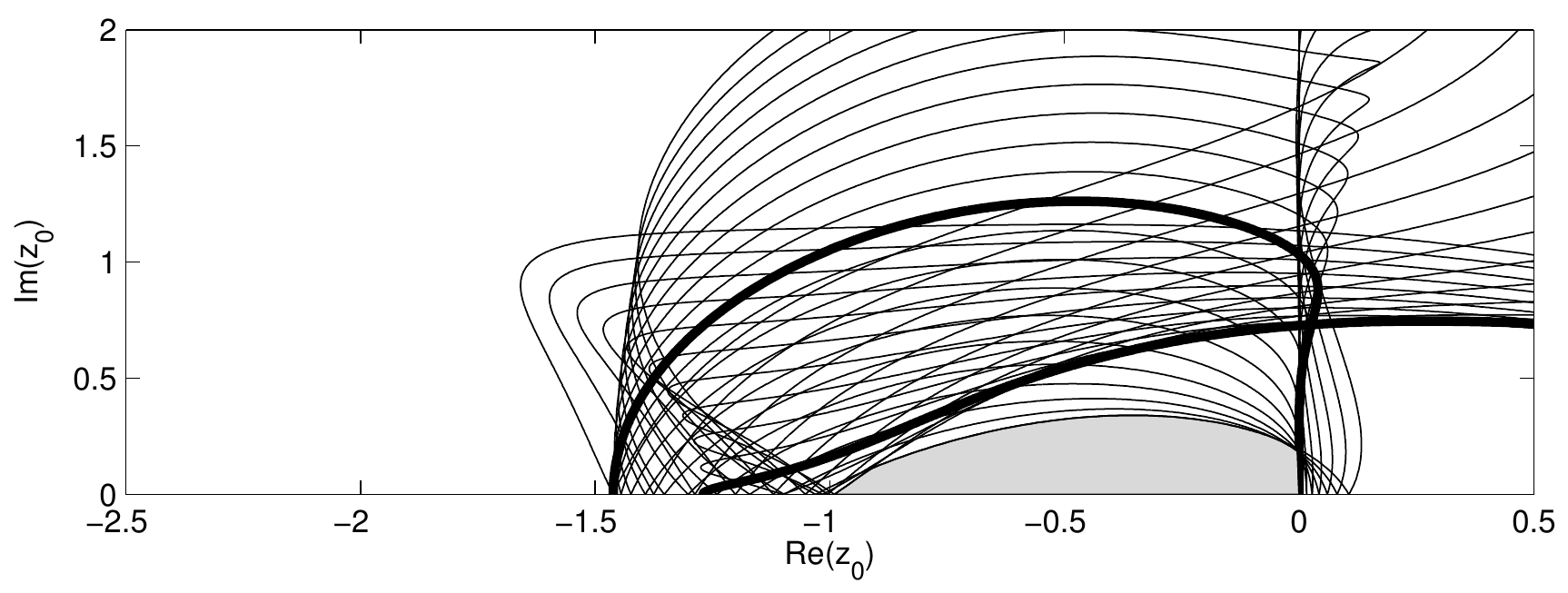}
\caption{Stability regions $\mathcal{S}_{\pi/2,y}$, $y=-2.0,-1.8,\ldots,2.0$ (thin lines),
$\mathcal{S}_{\pi/2}$ (shaded region), and $\mathcal{S}_E$ (thick line)
for $\beta_{21}\approx 1.39$, $\beta_{31}\approx -0.146$, and $\beta_{32}\approx 1.24$} \label{fig5.10}
\end{center}
\end{figure}
\begin{figure}[t!h!b!]
\begin{center}
\includegraphics[width=0.75\textwidth]{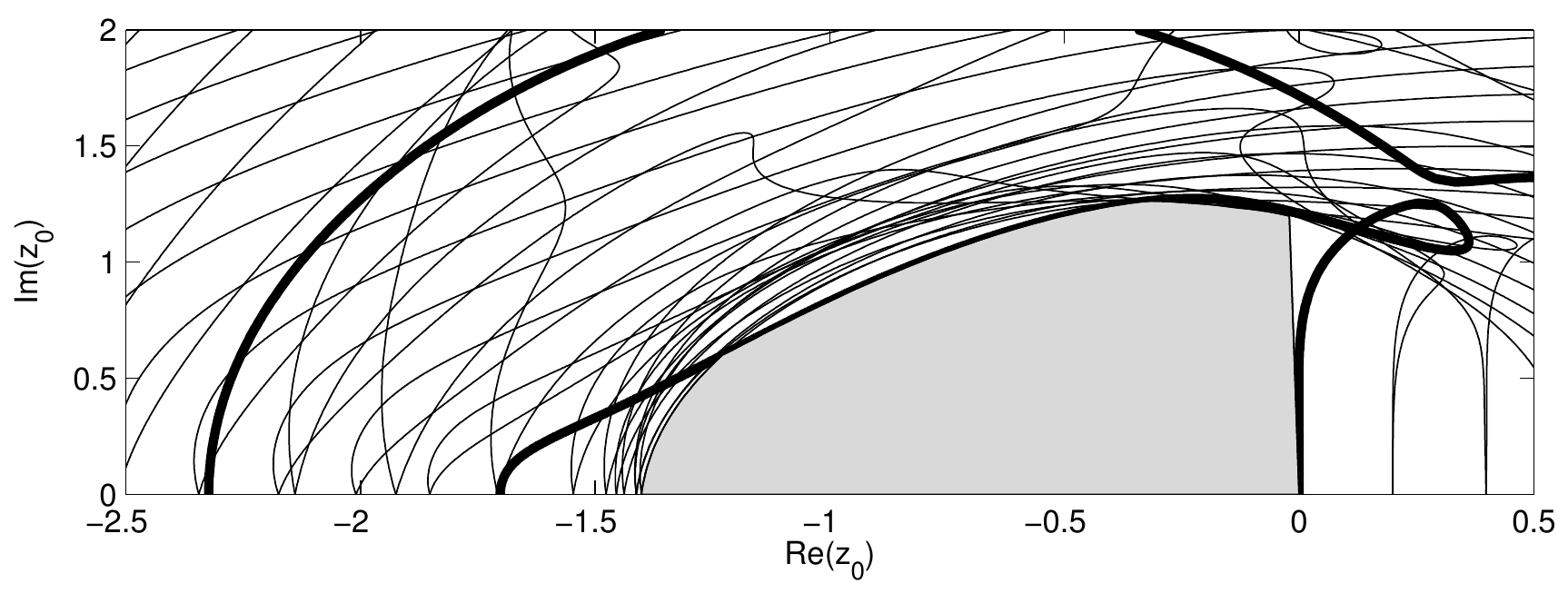}
\caption{Stability regions $\mathcal{S}_{\pi/4,y}$, $y=-2.0,-1.8,\ldots,2.0$ (thin lines),
$\mathcal{S}_{\pi/4}$ (shaded region), and $\mathcal{S}_E$ (thick line)
for $\beta_{21}\approx 1.25$, $\beta_{31}\approx 1.62$, and $\beta_{32}\approx 0.0555$} \label{fig5.11}
\end{center}
\end{figure}

We have performed a computer search in the parameter space
$\beta_{21}$, $\beta_{31}$, and $\beta_{32}$ looking first for methods
for which the stability region $\mathcal{S}_E$ of the explicit method
is maximal. This corresponds to the parameter values
$
\beta_{21}\approx 1.13
$,
$
\beta_{31}\approx 1.45
$,
$
\beta_{32}\approx -0.158
$,
for which the area of $\mathcal{S}_E$ is approximately equal to $3.54$.
The stability region $\mathcal{S}_E$ of the resulting method is plotted
on Fig.~\ref{fig5.8} by a thick line. We have also plotted stability regions
$\mathcal{S}_{\pi/2,y}$ for $y=-2.0,-1.8,\ldots,2.0$ (thin lines) and the stability
region $\mathcal{S}_{\pi/2}$ (shaded region) of the corresponding IMEX scheme.
We can see that this region $\mathcal{S}_{\pi/2}$ is substantially smaller
than the region $\mathcal{S}_E$, the area of $\mathcal{S}_{\pi/2}$ is approximately
equal to $0.39$.
However, we can obtain larger regions $\mathcal{S}_{\alpha}$ for values of
$\alpha$ smaller than $\pi/2$, i.e., if we relax the requirement that the
implicit part of IMEX scheme is $A$-stable and require instead
$A(\alpha)$-stability for $\alpha<\pi/2$. This is illustrated on
Fig.~\ref{fig5.9}, where we have plotted again the stability region
$\mathcal{S}_E$ of the explicit method (thick line), stability regions
of $\mathcal{S}_{\alpha,y}$ for $y=-2.0,-1.8,\ldots,2.0$ (thin lines),
and stability region $\mathcal{S}_{\alpha}$ (shaded region) for $\alpha=\pi/4$.
The area of this region is approximately equal to $1.91$.

We have also performed a computer search looking for methods for
which stability regions $\mathcal{S}_{\alpha}$ are maximal for some
fixed values of $\alpha$. For $\alpha=\pi/2$ this corresponds to
the parameter values $\beta_{21}\approx 1.39$, $\beta_{31}\approx -0.146$,
and $\beta_{32}\approx 1.24$ for which the area of $\mathcal{S}_{\pi/2}$
is approximately equal to $0.50$.
For $\alpha=\pi/4$ this corresponds to
the parameter values $\beta_{21}\approx 1.25$, $\beta_{31}\approx 1.62$,
and $\beta_{32}\approx 0.00555$ for which the area of $\mathcal{S}_{\pi/4}$
is approximately equal to $2.80$.
We have plotted on Fig.~\ref{fig5.10} and Fig.~\ref{fig5.11}
the stability regions $\mathcal{S}_E$ of the resulting explicit
methods (thick lines), the regions $\mathcal{S}_{\alpha,y}$ for
$y=-2.0,-1.8,\ldots,2.0$ (thin lines) and stability regions $\mathcal{S}_{\alpha}$
of IMEX schemes (shaded regions) for $\alpha=\pi/2$ and $\alpha=\pi/4$.

\subsection{IMEX GLMs with $p=q=r=s=4$} \label{sec5.4}
Let $\lambda\approx 0.57281606$ be a root of the polynomial
$$
\varphi(\lambda)=\lambda^4-4\lambda^3+3\lambda^2-\ds\frac{2}{3}\lambda+\ds\frac{1}{24},
$$
and consider the implicit DIMSIM with $\mbf{c}=[0,1/3,2/3,1]^T$, the coefficient
matrix $\mbf{A}$ given by
$$
\mbf{A}=\left[
\begin{array}{cccc}
0.57281606 & \phantom{-}0 & 0 & 0 \\
0.15022075 & \phantom{-}0.57281606 & 0 & 0 \\
0.59515808 & -0.26632807 & 0.57281606 & 0 \\
1.7717286  & -1.64234444 & 0.39147320 & 0.57281606
\end{array}
\right],
$$
the rank one coefficient matrix $\mbf{V}=\mbf{e}\mbf{v}^T$,
where $\mbf{e}=[1,1,1,1]^T$ and
$$
\mbf{v}=\left[
\begin{array}{cccc}
15.615037 & -46.967269 & 41.290082 & -8.9378502
\end{array}
\right]^T,
$$
and the vectors $\mbf{q}_0$, $\mbf{q}_1$, $\mbf{q}_2$, $\mbf{q}_3$, and $\mbf{q}_4$
equal to $\mbf{q}_0=\mbf{e}$,
$$
\mbf{q}_1=\left[
\begin{array}{cccc}
-0.57281606 &
-0.38970348 &
-0.23497940 &
-0.093673420
\end{array}
\right]^T,
$$
$$
\mbf{q}_2=\left[
\begin{array}{cccc}
0 &
-0.13538313 &
-0.070879128 &
0.21364995
\end{array}
\right]^T,
$$
$$
\mbf{q}_3=\left[
\begin{array}{cccc}
0 &
-0.025650275 &
-0.063113738 &
-0.11549405
\end{array}
\right]^T,
$$
$$
\mbf{q}_4=\left[
\begin{array}{cccc}
0 &
-0.0030214983 &
-0.018412760 &
-0.062996758
\end{array}
\right]^T.
$$
Computing the coefficient matrix $\mbf{B}$ from the relation
(\ref{eq5.1}) leads to the method of order $p=4$ and stage order $q=4$.
It was demonstrated in \cite{wri01} that the resulting method is
$A$- and $L$-stable.

\begin{figure}[t!h!b!]
\begin{center}
\includegraphics[width=0.75\textwidth]{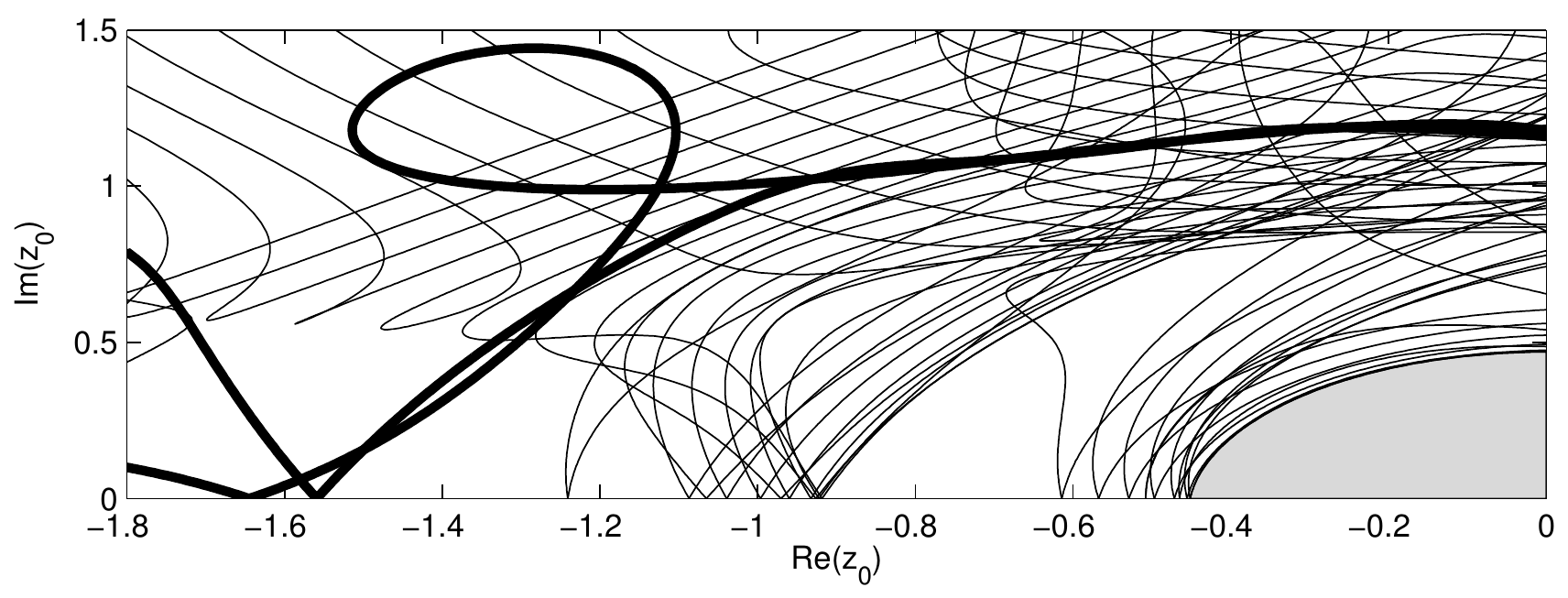}
\caption{Stability regions $\mathcal{S}_{\pi/4,y}$, $y=-2.0,-1.8,\ldots,2.0$ (thin lines),
$\mathcal{S}_{\pi/4}$ (shaded region), and $\mathcal{S}_E$ (thick line)
for $\beta_{21}\approx 0.0645$, $\beta_{31}\approx -0.351$, and $\beta_{32}\approx 0.272$,
$\beta_{41}\approx -2.82$, $\beta_{42}\approx 3.47$, $\beta_{43}\approx -1.05$} \label{fig5.12}
\end{center}
\end{figure}
\begin{figure}[t!h!b!]
\begin{center}
\includegraphics[width=0.75\textwidth]{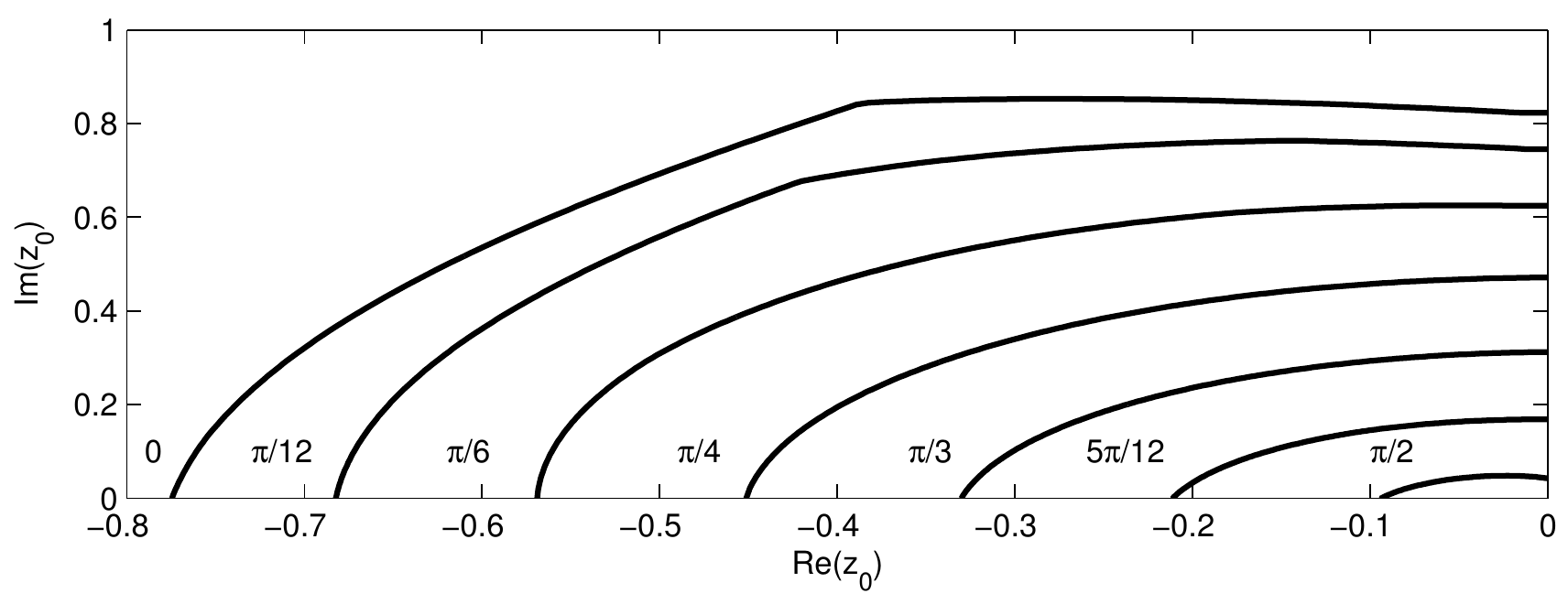}
\caption{Stability regions $\mathcal{S}_{\alpha}$ for $\alpha=0$, $\pi/12$, $\pi/6$,
$\pi/4$, $\pi/3$, $5\pi/12$, and $\pi/2$
for $\beta_{21}\approx 0.0645$, $\beta_{31}\approx -0.351$, and $\beta_{32}\approx 0.272$,
$\beta_{41}\approx -2.82$, $\beta_{42}\approx 3.47$, $\beta_{43}\approx -1.05$} \label{fig5.13}
\end{center}
\end{figure}
\begin{figure}[t!h!b!]
\begin{center}
\includegraphics[width=0.75\textwidth]{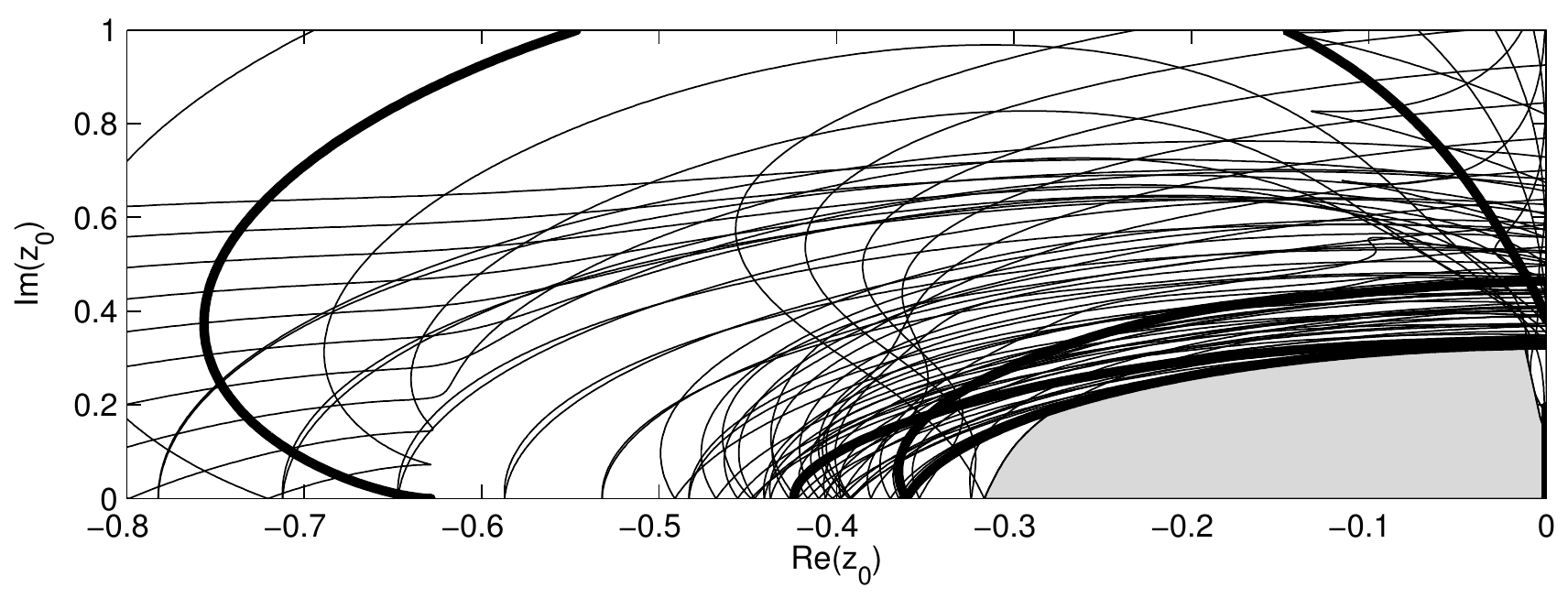}
\caption{Stability regions $\mathcal{S}_{\pi/2,y}$, $y=-2.0,-1.8,\ldots,2.0$ (thin lines),
$\mathcal{S}_{\pi/2}$ (shaded region), and $\mathcal{S}_E$ (thick line)
for $\beta_{21}\approx -0.00516$, $\beta_{31}\approx -0.939$,
$\beta_{32}\approx 1.18$, $\beta_{41}\approx -1.71$, $\beta_{42}\approx 2.07$,
and $\beta_{43}\approx 0.32$} \label{fig5.14}
\end{center}
\end{figure}
\begin{figure}[t!h!b!]
\begin{center}
\includegraphics[width=0.75\textwidth]{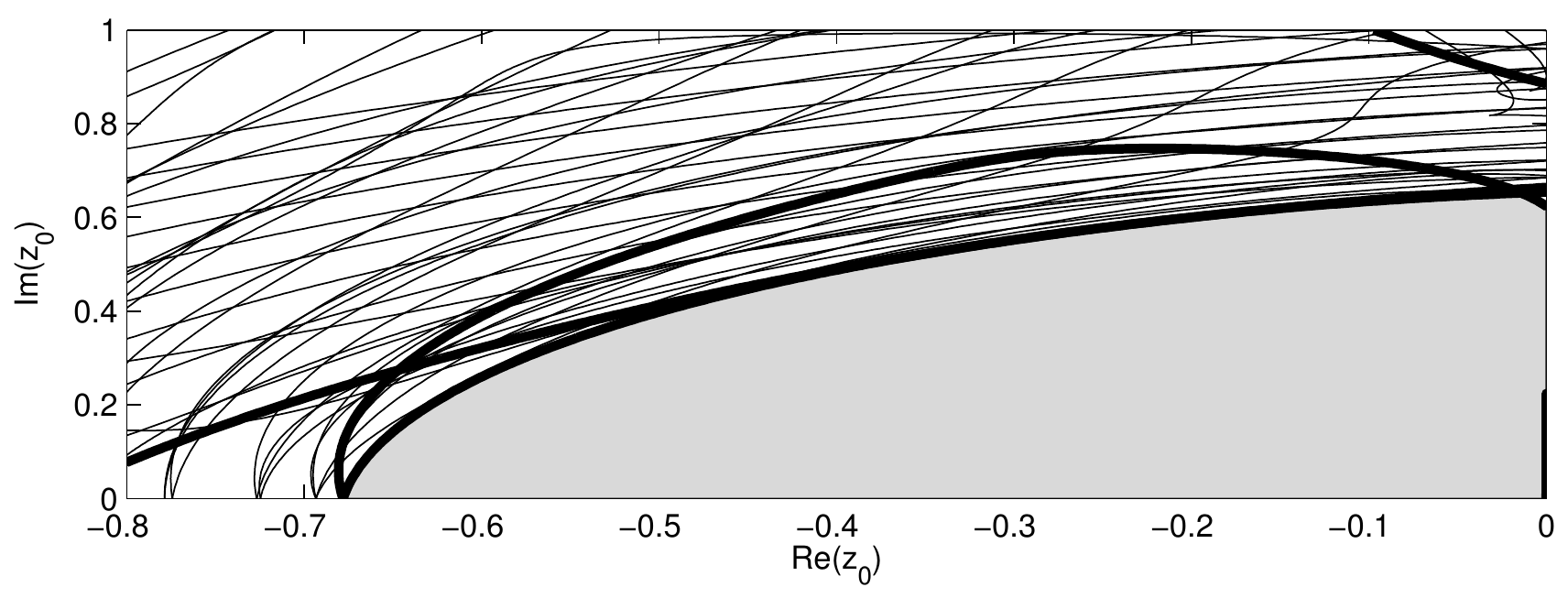}
\caption{Stability regions $\mathcal{S}_{\pi/4,y}$, $y=-2.0,-1.8,\ldots,2.0$ (thin lines),
$\mathcal{S}_{\pi/4}$ (shaded region), and $\mathcal{S}_E$ (thick line)
for $\beta_{21}\approx 0.0964$, $\beta_{31}\approx -0.278$,
$\beta_{32}\approx 0.464$, $\beta_{41}\approx -1.63$, $\beta_{42}\approx 2.73$,
and $\beta_{43}\approx -0.678$} \label{fig5.15}
\end{center}
\end{figure}

The coefficients $\alpha_{jk}$ of the extrapolation formula (\ref{eq1.11})
of order $p=4$ computed from the system (\ref{eq2.12}) corresponding to
$p=s=4$ take the form
$$
\alpha=\left[
\begin{array}{cccc}
0 & 0 & 0 & 1 \\
-1 & 4 & -6 & 4-\beta_{21} \\
\beta_{32}-4 & 15-4\beta_{32} & 2(3\beta_{32}-10) & 10-\beta_{31}-4\beta_{32} \\
\alpha_{41} & \alpha_{42} & \alpha_{43} & \alpha_{44}
\end{array}
\right],
$$
with
$$
\begin{array}{ll}
\alpha_{41}=\beta_{42}+4\beta_{43}-10, &
\alpha_{42}=36-4\beta_{42}-15\beta_{43}, \\
\alpha_{43}=6\beta_{42}+20\beta_{43}-45, &
\alpha_{44}=20-\beta_{41}-4\beta_{42}-10\beta_{43}.
\end{array}
$$

To investigate the stability properties of the resulting IMEX scheme
we will work with the stability polynomial
$p(w,z_0,z_1)$ obtained by multiplying stability function of
the method by a factor $(1-\lambda z_1)^4$,
where $\lambda$ is the diagonal element of the matrix $\mbf{A}$.
It can be verified that
this polynomial takes the form
$$
\begin{array}{l}
p(w,z_0,z_1) \ = \
(1-\lambda z_1)^4w^8-p_7(z_0,z_1)w^7
+p_6(z_0,z_1)w^6-p_5(z_0,z_1)w^5 \\ [2mm]
\quad\quad + \
p_4(z_0,z_1)w^4-p_3(z_0,z_1)w^3+p_2(z_0,z_1)w^2-p_1(z_0,z_1)w+p_0(z_0,z_1),
\end{array}
$$
with coefficients $p_0(z_0,z_1)$, $p_1(z_0,z_1)$, $p_2(z_0,z_1)$,
$p_3(z_0,z_1)$, $p_4(z_0,z_1)$, \mbox{$p_5(z_0,z_1)$}, $p_6(z_0,z_1)$, and $p_7(z_0,z_1)$ which
are polynomials of degree less than or equal to $4$ with respect to
$z_0$ and $z_1$. These coefficients
depend also on $\beta_{21}$, $\beta_{31}$, $\beta_{32}$, $\beta_{41}$,
$\beta_{42}$, and $\beta_{43}$.

We have performed a computer search in the parameter space
$\beta_{21}$, $\beta_{31}$, $\beta_{32}$, $\beta_{41}$, $\beta_{42}$,
and $\beta_{43}$ looking first for methods
for which the stability region $\mathcal{S}_E$ of the explicit method
is maximal. This corresponds to the parameter values
$
\beta_{21}\approx 0.0625
$,
$
\beta_{31}\approx -0.355
$,
$
\beta_{32}\approx 0.272
$,
$
\beta_{41}\approx -2.84
$,
$
\beta_{42}\approx 3.49
$,
$
\beta_{43}\approx -1.06
$,
for which the area of $\mathcal{S}_E$ is approximately equal to $2.82$.
The stability region $\mathcal{S}_E$ of the resulting method is plotted
on Fig.~\ref{fig5.12} by a thick line. We have also plotted stability regions
$\mathcal{S}_{\pi/4,y}$ for $y=-2.0,-1.8,\ldots,2.0$ (thin lines) and the stability
region $\mathcal{S}_{\pi/4}$ (shaded region) of the corresponding IMEX scheme.
The area of $\mathcal{S}_{\pi/4}$ is approximately equal to $0.32$.
We have also plotted on Fig.~\ref{fig5.13} stability regions
$\mathcal{S}_{\alpha}$ for $\alpha=0$, $\pi/12$, $\pi/6$,
$\pi/4$, $\pi/3$, $5\pi/12$, and $\pi/2$ corresponding to the same values
of $\beta_{ij}$.
We can see in particular that the region $\mathcal{S}_{\pi/2}$ is quite small, its area is
approximately equal to $0.0069$.

As in Section~\ref{sec5.3} we have also performed a computer search looking
directly for methods for
which stability regions $\mathcal{S}_{\alpha}$ are maximal for some
fixed values of $\alpha$. For $\alpha=\pi/2$ this corresponds to
the parameter values $\beta_{21}\approx -0.00516$, $\beta_{31}\approx -0.939$,
$\beta_{32}\approx 1.18$, $\beta_{41}\approx -1.71$, $\beta_{42}\approx 2.07$,
and $\beta_{43}\approx 0.32$, for which the area of $\mathcal{S}_{\pi/2}$
is approximately equal to $0.16$.
For $\alpha=\pi/4$ this corresponds to
the parameter values $\beta_{21}\approx 0.0964$, $\beta_{31}\approx -0.278$,
$\beta_{32}\approx 0.464$, $\beta_{41}\approx -1.63$, $\beta_{42}\approx 2.73$,
and $\beta_{43}\approx -0.678$, for which the area of $\mathcal{S}_{\pi/4}$
is approximately equal to $0.65$.
We have plotted on Fig.~\ref{fig5.14} and Fig.~\ref{fig5.15}
the stability regions $\mathcal{S}_E$ of the resulting explicit
methods (thick lines), the regions $\mathcal{S}_{\alpha,y}$ for
$y=-2.0,-1.8,\ldots,2.0$ (thin lines) and stability regions $\mathcal{S}_{\alpha}$
of IMEX schemes (shaded regions) for $\alpha=\pi/2$ and $\alpha=\pi/4$.

\section{Numerical experiments} \label{sec:numerics}

\setcounter{equation}{0}
\setcounter{figure}{0}
\setcounter{table}{0}

The extrapolation-based IMEX GLMs constructed in Section \ref{sec:construction} have been implemented in Matlab.
The required starting values $y^{[0]}$ and $Y^{[0]}$ were computed by finite difference approximations from solutions obtained with the Matlab routine \texttt{ode15s}.

The test problem is the two dimensional shallow-water equations system \cite{swe1998},
which approximates a thin layer of fluid inside a shallow basin:
\begin{eqnarray}
 \frac{\partial}{\partial t} h + \frac{\partial}{\partial x} (uh) + \frac{\partial}{\partial y} (vh) &=& 0 \nonumber \\
 \frac{\partial}{\partial t} (uh) + \frac{\partial}{\partial x} \left(u^2 h + \frac{1}{2} g h^2\right) + \frac{\partial}{\partial y} (u v h) &=& 0  \label{swe} \\
 \frac{\partial}{\partial t} (vh) + \frac{\partial}{\partial x} (u v h) + \frac{\partial}{\partial y} \left(v^2 h + \frac{1}{2} g h^2\right) &=& 0 \;.
\nonumber
\end{eqnarray}
Here $h(t,x,y)$ is the fluid layer thickness, $u(t,x,y)$ and $v(t,x,y)$ are the components of the velocity field,
and $g$ denotes the gravitational acceleration.
The spatial domain is $\Omega = [-3,\,3]^2$ (spatial units), and the integration window is $t_0 = 0 \le t \le t_\textrm{f} = 10$ (time units).
We use reflective boundary conditions and the initial conditions at $t_0 = 0$
\begin{equation}
u(t_0,x,y)=v(t_0,x,y)=0\,,~~
h(t_0,x,y) = 1 + e^{-\|(x,y)-(c_1,c_2)\|^2_2}\,,
\end{equation}
with the Gaussian height profile $c_1 = 1/3$ and $c_2 = 2/3$.

A second order Lax-Wendroff finite difference scheme is used for space discretization, resulting in a semi-discrete ODE system of the form
\begin{equation}
\frac{\mathrm{d}}{\mathrm{d}t}U(t) = F \left(U\right) = \underbrace{ F_U \bigl(U\bigr)\cdot U(t) }_{g(U)}+  \underbrace{\left( F(U) - F_U \bigl(U\bigr)\cdot U(t) \right)}_{f(U)} \, ,
\label{semiDescreteODEswe}
\end{equation}
where $U(t)$ is a combined column vector of the discretized state variables $ (\widehat{h},\widehat{uh},\widehat{vh})$,
and $F_U=\partial F/\partial U$ be the Jacobian of right hand side.
We consider a splitting of equation \eqref{semiDescreteODEswe} into the linear part $g(U)$, considered stiff, and the nonlinear part $f(U)$,
considered non-stiff.
The linear stiff part is treated implicitly, and the non-stiff part is treated explicitly.

We compare the numerical results for the solution at the final time with a reference solution computed by the Matlab function \texttt{ode15s} with very tight tolerances $atol=rtol=10^{-14}$.
The errors are measured in $\mathcal{L}_2$ norms.
The error diagram against the time step size is presented in Fig. \ref{fig6.1}.
The observed orders for all the methods tested match the theoretical predictions.

 \begin{figure}[t!h!b!]
 \begin{center}
 \includegraphics[width=0.70\textwidth]{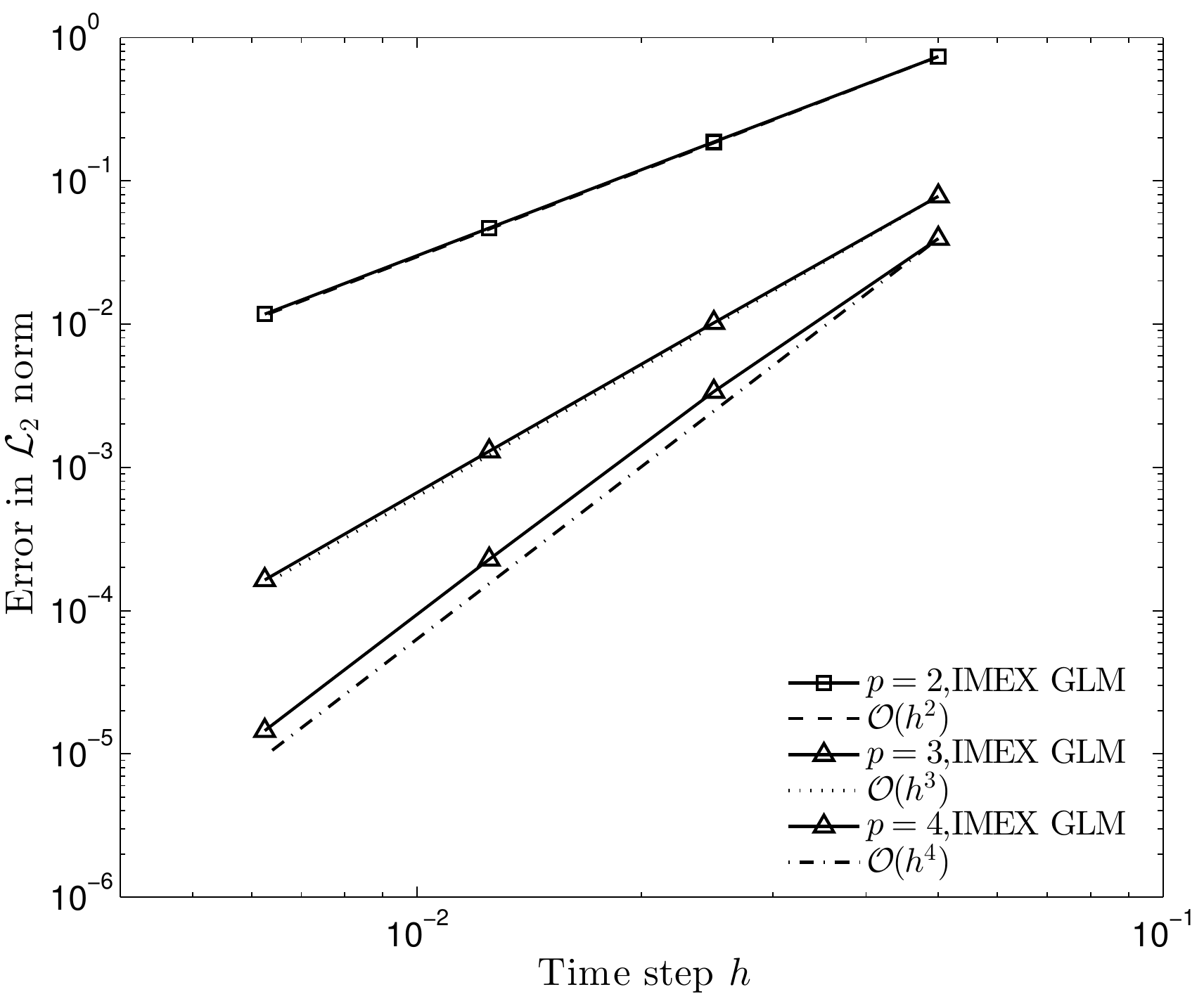}
 \caption{Error vs. time step size for several extrapolation-based IMEX-GLMs applied to the shallow water equations test problem.} \label{fig6.1}
 \end{center}
 \end{figure}

\section{Concluding remarks} \label{sec:conclusions}

General linear methods offer an excellent framework for the construction of implicit-explicit schemes.
In this paper we develop a new extrapolation-based approach for the construction of practical IMEX GLM schemes of high order and high stage order.
The accuracy, linear stability, and Prothero-Robinson convergence are analyzed.
These schemes  are particularly attractive for solving stiff problems, where other multistage methods may suffer from order reduction.

The extrapolation-based mechanism offers a systematic approach for constructing IMEX GLM schemes.
The construction starts with the selection of an implicit component method with suitable stability and order properties. The explicit component is then obtained
though an optimization procedure that maximized the combined stability region of the pair.  We apply this methodology to construct IMEX pairs of orders one to four.

Future work is planned to extend the extrapolation idea to construct other types of partitioned GLMs, including parallel time integrators, and asynchronous pairs of methods
that do not share the same abscissae.

\textbf{Acknowledgements.}
The results reported in this paper were obtained during the
visit of the first author to the Arizona State University
in January--March of $2013$. This author wish to
express her gratitude to the School
of Mathematical and Statistical Sciences for hospitality during this visit.
The work of Sandu and Zhang has been supported in part by the awards NSF OCI-8670904397, NSF CCF-0916493, NSF DMS-0915047, NSF CMMI-1130667, NSF CCF-1218454, AFOSR FA9550-12-1-0293-DEF, AFOSR 12-2640-06, and by the Computational Science Laboratory at Virginia Tech.

\end{document}